\documentclass[11pt]{amsart}
 \usepackage{geometry}                
\usepackage{graphicx}
\usepackage{amssymb}
\usepackage{epstopdf}
\usepackage{color}
\newtheorem{theorem}{Theorem}[section]
\newtheorem{lemma}[theorem]{Lemma}

\newtheorem{proposition}[theorem]{Proposition}

\theoremstyle{definition}

\theoremstyle{remark}
\newtheorem{remark}[theorem]{Remark}

\theoremstyle{example}

\DeclareFontFamily{OML}{rsfs}{\skewchar\font'177}
\DeclareFontShape{OML}{rsfs}{m}{n}{ <5> <6> rsfs5 <7> <8> <9>
rsfs7 <10> <10.95> <12> <14.4> <17.28> <20.74> <24.88> rsfs10 }{}
\DeclareMathAlphabet{\mathfs}{OML}{rsfs}{m}{n}

\newcommand{\BH}{{\mathbb{H}}}

\newcommand{\BR}{{\mathbb{R}}}

\newcommand{\BZ}{{\mathbb{Z}}}

\newcommand{\CB}{{\mathcal{B}}}

\newcommand{\CD}{{\mathcal{D}}}
\newcommand{\CE}{{\mathcal{E}}}

\newcommand{\prob}{{\bf P}}
\newcommand{\p}{{\bf P}}
\newcommand{\e}{{\bf E}}
\newcommand{\bae}{\begin{equation}\begin{aligned}}
\newcommand{\eae}{\end{aligned}\end{equation}}
\newcommand{\ev}{\mathbf{E}}
\newcommand{\pr}{\mathbb{P}}

\newcommand{\htr}[1]{{\color{black}{{h(T(#1))}}}}

\DeclareMathOperator*{\esssup}{ess\,sup}

\begin{document}

\title[Stationary Eden Model on groups]{Stationary Eden Model on groups}
\author{Ton\'ci Antunovi\'c}
\address[Ton\'ci Antunovi\'c]{University of California Los-Angeles}
\urladdr{http://www.math.ucla.edu/$\sim$tantunovic}
\email{tantunovic@math.ucla.edu}
\author{Eviatar B. Procaccia}
\address[Eviatar B. Procaccia\footnote{Research supported by NSF grant 1407558}]{University of California Los-Angeles}
\urladdr{www.math.ucla.edu/$\sim$procaccia}
\email{procaccia@math.ucla.edu}

\maketitle


\begin{abstract}
We consider two stationary versions of the Eden model, on the upper half planar lattice, resulting in an infinite forest covering the half plane. Under weak assumptions on the weight distribution and by relying on ergodic theorems, we prove that almost surely all trees are finite. 
Using the mass transport principle, we generalize the result to Eden model in graphs of the form $G\times\BZ_+$, where $G$ is a Cayley graph.
This generalizes certain known results on the two-type Richardson model, in particular of Deijfen and H{\"a}ggstr{\"o}m in 2007 \cite{deijfen2007two}.
\end{abstract}

\numberwithin{equation}{section} 

\section{Introduction}
Aggregation processes form one of the richest class of processes in statistical physics, admitting examples of KPZ relations, fractal geometry and surprising scaling limits. Though easy to define, many of the aggregation processes defy rigorous analysis. Itai Benjamini suggested to study a stationary version of known aggregation processes. The idea is to let aggregation processes grow from an infinite base graph, instead from a single point, and result in an infinite forest rooted at the base graph. The stationary versions share local behavior with the known processes, giving us a new approach to study them.   
A first attempt was made by Berger, Kagan and Procaccia \cite{berger2012stretch}, where a stationary version of internal diffusion limited aggregation (SIDLA) was studied on the upper half planar lattice. The general philosophy of the project is to use the additional symmetries given by stationarity to obtain local behavior of aggregation processes.

The Eden model was defined by Murray Eden in 1961 \cite{eden}. Consider the lattice $\BZ^d$ with the set of edges $\CE$. The Eden Model is commonly defined as a stochastic process with the state space  $\{0,1\}^\CE$, supported on finite nearest neighbor connected sets. For $1\le i\le d$, let $\pr[A_1=\pm e_i]=(2d)^{-1}$, where $e_i$ are the standard lattice coordinate directions. Conditioned on $A_n$, let $\partial A_n$ be the edge boundary of $A_n$, and let $\pr[A_{n+1}=A_n\cup \{e\}]=|\partial A_n|^{-1},$ for every $e\in\partial A_n$.

In Lawler, Bramson and Griffeath 1992, \cite{idla} it is claimed that computer simulations suggest that the Eden model does not converge to a Euclidean ball (proved for dimension greater than $10^6$ in Kesten 1986 \cite[Corollary 8.4]{kesten1986aspects}). It seems that even though the Eden model appears to be the simplest aggregation process, it holds surprising geometric properties.


In this paper we study a version of the Eden model on the graphs of the form $G \times \mathbb{Z}_+$, where $G$ is a Cayley graph of an infinitely countable, finitely generated group.  
For simplicity, we begin the discussion with the simplest case $G = \mathbb{Z}$, when $G \times \mathbb{Z}_+$ corresponds to either the upper part or the upper-right part of the square planar lattice.
The proof of this case contains the key ideas of the general case, while avoiding technical encumbrances. 

We consider two variants of the model. 
For the simplest case $G= \mathbb{Z}$, let $\overline{\BH}=\BZ\times\BZ_+$ to be the upper half planar lattice with nearest neighbor edges $\overline{\CE} $, and let $\widehat{\BH}$ be the half planar directed lattice $\widehat{\BH}=\{(x,y):x,y \in \BZ, x+y\in2\BZ,y\ge0\}$, with directed edges $\widehat{\CE}=\{(x+\theta_l,x),(x+\theta_r,x):x\in\widehat{\BH}\}$, where $\theta_l=(-1,1)$ and $\theta_r=(1,1)$. 
In other words, the edges are directed towards the vertex which decreases the sum of coordinates.

In the general case, we will take $G$ to be a Cayley graph of finitely generated group with a finite generating set.
In either the directed or undirected case we will consider graphs whose set of vertices is $G \times \mathbb{Z}_+$. 
In the directed case we consider the oriented graph $\widehat{G}$ with the vertex set $G \times \mathbb{Z}_+$, and in which vertices $(x,m)$ and $(y,n)$ are connected by an edge if and only if $|m-n|=1$ and $x$ and $y$ are neighbors in $G$.
This is also known as the tensor product of the graphs $G$ and $\mathbb{Z}_+$.
Edge connecting $(x,n)$ and $(y,n+1)$ is given the orientation from $(y,n+1)$ to $(x,n)$. 
In the undirected case, we consider the unoriented graph $\overline{G}$ with the vertex set $G \times \mathbb{Z}_+$, in which vertices $(x,m)$ and $(y,n)$ are connected by an edge if either 
\begin{itemize}
\item $x$ and $y$ are neighbors in $G$ and $m = n \geq 1$, or
\item $x = y$ and $m,n \geq 0$ with $|m-n| =1$.
\end{itemize}
This is equivalent to taking a Cartesian product of graphs $G \times \mathbb{Z}^+$ and removing the edges (but not the vertices) of $G \times \{0\}$.

Note that if we make all the edges in $\widehat{G}$ unoriented, the graph $\widehat{G}$ will be connected if and only if $G$ is not bipartite. 
If $G$ is bipartite it will actually consists of two disjoint connected components, isomorphic through the mapping $(x,n) \mapsto (zx,n)$ for any fixed generator $z \in G$.
In particular, taking $G = \mathbb{Z}$, the graph  $\widehat{\mathbb{H}}$ is actually one of the  components of the graph $\widehat{G}$, while $\overline{\mathbb{H}}$  agrees with the constructed $\overline{G}$.

The edge sets of $\widehat{G}$ and $\overline{G}$ will be denoted by $\widehat{\mathcal{E}}$ and $\overline{\mathcal{E}}$, respectively.
In the general case, will use symbols $x,y,z \dots$ to denote the vertices of the graph $G$ (that is the elements of the group $\mathcal{G}$), while $\overline{x}, \overline{y}, \overline{z} \dots$ will denote the vertices of $\overline{G}$, and $\widehat{x}, \widehat{y}, \widehat{z} \dots$ will denote the vertices of $\widehat{G}$.
Similarly, while $e$ and $\gamma$ will denote an edge and a path of $G$, $\overline{e}$, $\overline{\gamma}$  will denote edges and paths of $\overline{G}$, and $\widehat{e}$, $\widehat{\gamma}$  will denote (directed) edges and paths of $\widehat{G}$.
Note that we will consider paths as either sequences of adjacent vertices or sequences of adjacent  edges, as we find convenient.
For pairs of vertices $(x,y)$, $(\overline{x}, \overline{y})$ and $(\widehat{x}, \widehat{y})$ in the corresponding graphs, we let 
$\Gamma(x,y)$, $\overline{\Gamma}(\overline{x},\overline{y})$ and $\widehat{\Gamma}(\widehat{x},\widehat{y})$ denote the sets of paths connecting the corresponding vertices.
Note that paths in $\widehat{\Gamma}(\widehat{x},\widehat{y})$ are directed, and so $\widehat{\Gamma}(\widehat{x},\widehat{y}) = \emptyset$ is possible.

Note that the graph $G$ naturally embeds into both $\overline{G}$ and $\widehat{G}$ as $x \mapsto (x,0)$, and thus vertices $(x,0)$  will be denoted simply by $x$. 
We will denote the image $\{(x,0): x \in G\}$ as $G$.
Analogously, we will denote $\partial \widehat{\BH} = (2\mathbb{Z}) \times \{0\}$ and $\partial \overline{\BH} = (\mathbb{Z}) \times \{0\}$.

Let $\mu$ be a non-atomic measure supported on $[0,\infty)$ with finite expectation $\int_{[0,\infty)}x\mu(dx)<\infty$. 
To edges of $\widehat{G}$ and $\overline{G}$ (that is $\widehat{\BH}$ and $\overline{\BH}$ in the one-dimensional case) we will assign i.i.d.~ random variables $(\omega(\widehat{e}))_{\widehat{e} \in \widehat{\mathcal{E}}}$ and $(\omega(\overline{e}))_{\overline{e} \in \overline{\mathcal{E}}}$ with the distribution $\mu$.
The corresponding product measures over all edges is denoted by $\overline{\prob} = \mu^{\overline{\CE}}$ for the graph $\overline{G}$, and $\widehat{\prob} = \mu^{\widehat{\CE}}$ for the graph $\widehat{G}$. 

For every path $\overline{\gamma}$ in $\overline{G}$ and every directed path $\widehat{\gamma}$ in $\widehat{G}$ consider the passage times \[\lambda(\overline{\gamma}) = \sum_{\overline{e} \in \overline{\gamma}} \omega(\overline{e})  \ \text{ and } \  \lambda(\widehat{\gamma}) = \sum_{\widehat{e} \in \widehat{\gamma}} \omega(\widehat{e}),\] 
and the passage times between pairs of vertices $(\overline{x}, \overline{y})$ (or $(\widehat{x}, \widehat{y})$) 
\[
\overline{d}_\omega(\overline{x},\overline{y})=\inf_{\overline{\gamma}\in\overline{\Gamma}(\overline{x},\overline{y})}\lambda(\overline{\gamma})\ \text{ and } \ 
\widehat{d}_\omega(\widehat{x},\widehat{y})=\inf_{\widehat{\gamma}\in\widehat{\Gamma}(\widehat{x},\widehat{y})}\lambda(\widehat{\gamma}).
\]
Here $\widehat{d}_\omega(\widehat{x},\widehat{y}) = \infty$ if $\widehat{\Gamma}(\widehat{x},\widehat{y}) = \emptyset$.
For a set of vertices $A$ we define the point-to-set passage times $\overline{d}_\omega(\overline{x},A)=\inf_{\overline{y}\in A}\overline{d}_\omega(\overline{x},\overline{y})$ and  $\widehat{d}_\omega(\widehat{x},A)=\inf_{\widehat{y}\in A}\widehat{d}_\omega(\widehat{x},\widehat{y})$.
We will focus our interest on the passage times to $A = G$ and the geometry of the corresponding geodesics.
Note that since $\mu$ has no atoms for every pair of points $(\widehat{x}, \widehat{y})$, path $\widehat{\gamma}$ which achieves $\lambda(\widehat{\gamma}) = \widehat{d}_\omega(\widehat{x},\widehat{y})$ is unique, and this is also true for paths achieving $\widehat{d}_\omega(\widehat{x},A)$.
For every $x \in G$ and $t > 0$ define 
\[
\widehat{T}(x,t)=\bigcup_{\widehat{y}\in\widehat{G}} \{\widehat{\gamma}:\widehat{\gamma}\in\widehat{\Gamma}(\widehat{y},x),\lambda(\widehat{\gamma})=\widehat{d}_\omega(\widehat{y},G)<t\},
\]
as the union of all $\widehat{d}_\omega$ geodesics which end at  $x \in G$.
Similarly, we define $\overline{T}(x,t)$.
By the above discussion, for any fixed $t>0$ sets $(\widehat{T}(x,t))_{x \in G}$ (and similarly $(\overline{T}(x,t))_{x \in G}$) are disjoint trees.
We also consider the complete geodesic forest
\[
\widehat{T}(x)=\bigcup_{t > 0} \widehat{T}(x,t) = \bigcup_{\widehat{y}\in\widehat{G}} \{\widehat{\gamma}:\widehat{\gamma}\in\widehat{\Gamma}(\widehat{y},x),\lambda(\widehat{\gamma})=\widehat{d}_\omega(\widehat{y},G)\},
\]
and similarly $\overline{T}(x)$.

See Figure \ref{fig:directeden} for the visualization of the tree $\widehat{T}(0)$ in the directed one-dimensional case.
As discussed, for $G = \mathbb{Z}$ and more generally when $G$ is bipartite, the trees $\widehat{T}(x)$ spread through one of the two isomorphic components of the graph $\widehat{G}$.
We only draw one such component in Figure \ref{fig:directeden}.

\begin{figure}
\centering
\includegraphics[width=0.8\textwidth]{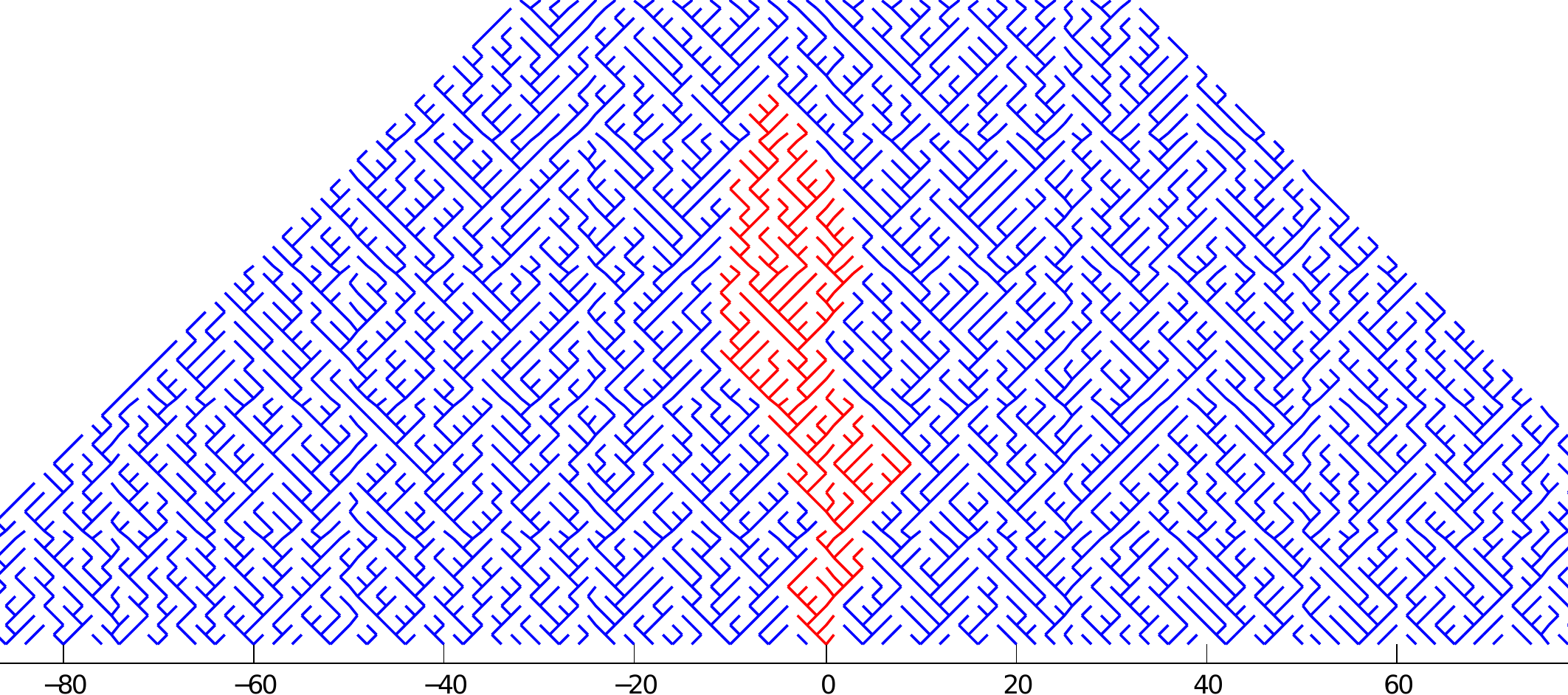}\label{fig:directeden}
\caption{Simulation of the Stationary Eden process on $\widehat{\BH}$.\label{fig:eden}}
\end{figure}

In the special case where $\mu$ is the distribution of an exponential random variable we obtain a stationary versions of the Eden model, see Figure \ref{fig:eden} for a graphical representation. By the memoryless property of the exponential distribution one can readily see that if we observe some tree $\overline{T}(x,t)$ ($\widehat{T}(x,t)$), at time $t$, the next edge the tree will attempt to add is uniform over the boundary of the tree. If at the time of attempt the end of the edge is not occupied by any tree, the edge will be added. If it is occupied the edge will not be  added. This is equivalent to the first passage percolation representation of the standard Eden model, where one considers all the geodesics emanating from the origin. This coupling was first considered by Richardson \cite{richardson1973random}, and was used by Kesten \cite{kesten1986aspects} to prove that the asymptotic shape of the Eden model in high dimension is not the Euclidean ball. 

The main result of this paper is that under the two defined measures all trees are finite almost surely. 
\begin{theorem}\label{thm:main}
For every non-atomic measure $\mu$, supported on $[0,\infty)$ with finite mean we have for any $x \in G$
\[\widehat{\prob}[|\widehat{T}(x)|<\infty]=1 \ \text{ and } \ \overline{\prob}[|\overline{T}(x)|<\infty]=1 .\]
\end{theorem}



In Deijfen and H{\"a}ggstr{\"o}m 2007 \cite[Theorem 1.1]{deijfen2007two} the undirected case of $G=\BZ^d$ with exponential weights was proved. 
In our proof we aspired greater generality for future applications in other stationary models.  

It is an interesting problem to describe the geometry of the trees more precisely.
In the following we prove that the maximal level size of any tree has infinite expectation and prove the same for certain moments of the tree heights.

For a set $S \subset \widehat{G}$ ($S \subset \overline{G}$) by $h(S)$ denote the height of $S$, that is
\[
h(S) = \max\{j: \exists x \in G, (x,j) \in S\},
\]
and by $w_n(S)$ and $w(S)$ denote the $n$-th and the maximal level size of $S$, that is
\[
w_n(S) = |S \cap G_n| \ \text{ and } w(S) = \max_{n \geq 0}w_n(S),
\]
where $G_n = G \times \{n\}$ represents the $n$-th level.
Note that when $G = \mathbb{Z}$ each level of $S$ is connected, so $w(S)$ can be interpreted as the maximal width of $S$, hence the notation.

\begin{theorem}
	\label{thm:finite_widhts}
        Assuming that $G$ is a Cayley graph of a finitely generated group, and that the distribution $\mu$ of $\omega(e)$ is supported on $[0,\infty)$ and has no atoms, the expected maximal level size has infinite expectation \[\widehat{\e}[w(\widehat{T}(x))] = \overline{\e}[w(\overline{T}(x))] = \infty.\]
\end{theorem}

\begin{theorem}\label{thm:finite_heights}
\begin{itemize}
\item[i)]
  Assuming that $G$ is a Cayley graph of a finitely generated group, and that the distribution $\mu$ of $\omega(e)$ is supported on $[0,\infty)$ and has no atoms, we have \[\widehat{\e}[\phi_G(h(\widehat{T}(x)))] =  \infty,\] where $\phi_G(n)$ is the size of the graph-distance ball of radius $n$ in $G$.
In particular, for $G = \mathbb{Z}$ we have that $\widehat{\e}[h(\widehat{T}(x))] =  \infty$.
\item[ii)]
For the undirected case, for $G = \mathbb{Z}^d$ and assuming additionally that 
 \bae\label{eq:kestenreq}
 \int e^{\nu x} \mu(dx)<\infty,\quad\text{for some $\nu>0$,}
 \eae  
we have that $\overline{\e}[h(\overline{T}(x))^d] =  \infty$.
\end{itemize}
\end{theorem}

Note that the total volumes of trees $|\widehat{T}(x)|$ and $|\overline{T}(x)|$ are at least as large as the maximal widths $w(\widehat{T}(x))$ and $w(\overline{T}(x))$, and so the expected volumes of trees are infinite as well. 
It would certainly be interesting to provide precise asymptotics for the maximal widths and heights. 
This seems to be rather difficult, even in dimension one.


Finally we show that the stationary Eden model converges asymptotically to a line, which is of interest due to the fact that the asymptotic shape of the Eden model is still unknown.
For a set $A \subset \overline{\BH}$, define the inner vertex boundary $\partial^{in}A$ as the set of all vertices of $A$ which have a neighbor in $A^c$.
\begin{theorem}\label{thm:shape}  
Suppose that $\mu$ is a non-atomic probability measure on $[0,\infty)$, which also satisfies \eqref{eq:kestenreq}.
Let $C_{\mathbf{d},t}$ denote the event that for all $(i,k)\in\left(\partial^{in}\bigcup_{x\in\BZ}\overline{T}(x,t)\right)$ such that $-t\le i\le t$, we have $\mathbf{d}-2t^{-0.1}\le\frac{k}{t}\le\mathbf{d}+2t^{-0.1}$ i.e.
\[
C_{\mathbf{d},t} = \left\{\left(\partial^{in}\bigcup_{x\in\BZ}\overline{T}(x,t)\right)\cap\left([-t,t]\times\BZ_+\right)\subset [-t,t]\times t[\mathbf{d}-2t^{-0.1},\mathbf{d}+2t^{-0.1}]\right\}.
\]
There exist constants $\mathbf{d}>0$ and $c>0$ such that for all $t>0$, we have  $\overline{\p}[C_{\mathbf{d},t}] \geq 1-e^{-ct^{1/5}}$. 
\end{theorem}

Since the arguments in the general case when $G$ is only assumed to be a Cayley graph are rather tedious, we will first present the arguments in the one-dimensional case $G = \mathbb{Z}$. 
Another reason for this approach is that a somewhat technical point  is resolved (in Lemma \ref{lemma:1dfpp}) in a simpler way compared to the general case, which might be of interest to some readers.
The one-dimensional case is presented in Section \ref{sec:One dimensional case}.
In Section \ref{sec:Notation and preliminary results} we will present some general results we will use in our arguments.
In Section \ref{sec:general} we present proofs for the general case.

\section{Notation and preliminary results}\label{sec:Notation and preliminary results}

\subsection{Notation}\label{subseq:notation} 
Throughout the paper we use the $\ \ \widehat{} \ \ $ and $\ \ \bar{} \ \ $ notation for the directed and undirected model respectively. 
For a result which holds in both models, we will omit the notation $\ \ \widehat{} \ \ $ and $\ \ \bar{} \ \ $ to make it context neutral.
This will be exploited quite often in the present section.

Let $\Lambda_n$ denote the subgraph of $\widehat{G}$ ($\overline{G}$) whose vertices are of the form  $(x,i)$ for $x \in G$ and $0 \leq i \leq n$, and which contains all edges between any two such  vertices.
We will use the same notation  $\Lambda_n$ in both directed and undirected cases to reduce the notation.
Let $\mathfs{F}_n = \sigma\{\omega(e) : e \in \Lambda_n\}$ denote the $\sigma$-algebra generated by the weights of edges in the first $n$ levels.
Let $\widehat{\gamma}(\widehat{x})$ and $\overline{\gamma}(\overline{x})$  be the (almost surely unique) geodesics from $\widehat{x}$ and $\overline{x}$ to $G$ respectively, that is the (random) path from $\widehat{x}$ and $\overline{x}$ to $G$ respectively, which minimizes passage times ($\lambda(\widehat{\tau})$ and $\lambda(\overline{\tau})$)  among all such paths ($\widehat{\tau}$ and $\overline{\tau}$).
Similarly denote the geodesics between vertices $\widehat{x}$ and $\widehat{y}$ (or $\overline{x}$ and $\overline{y}$) by $\widehat{\gamma}(\widehat{x},\widehat{y})$ (or $\overline{\gamma}(\overline{x},\overline{y})$).
Observe that $\widehat{y} \in \widehat{\gamma}(\widehat{x})$ ($y \in \overline{\gamma}(\overline{x})$) implies that $\widehat{\gamma}(\widehat{x},\widehat{y}) \subset \widehat{\gamma}(\widehat{x})$ ($\overline{\gamma}(\overline{x},\overline{y}) \subset \overline{\gamma}(\overline{x})$).
In either directed or undirected case, we will consider the edges in  the paths $\gamma(x)$ and $\gamma(x,y)$ to be ordered starting from $x$, and we will use the notation $\gamma_k(x)$ to denote the path consisting of the first $k$ edges of $\gamma(x)$.
Given the subgraph $\Lambda_n$ of either $\overline{G}$ or $\widehat{G}$, we can restrict the whole model to $\Lambda_n$.
In other words, for $\overline{x}, \overline{y} \in \Lambda_n$ we restrict the set of paths $\overline{\Gamma}(\overline{x},\overline{y})$ only to paths between $\overline{x}$ and $\overline{y}$ whose edges stay in $\Lambda_n$ (and similarly for the directed case).
In the undirected case, the analogues of $\overline{d}_{\omega}$, $\overline{T}(x)$, $\overline{\gamma}(\overline{x})$ and $\overline{\gamma}(\overline{x},\overline{y})$ will be denoted by $\overline{d}_{\omega,\Lambda_n}$, $\overline{T}_{\Lambda_n}(x)$, $\overline{\gamma}_{\Lambda_n}(\overline{x})$ and  $\overline{\gamma}_{\Lambda_n}(\overline{x},\overline{y})$. 
While the analogous notation can be use for the directed case, in the directed case this restriction is just an artifact of the model.
In the case of the directed lattice, let $\widehat{T}^n(x)$ denote the $n$-th level of the tree, that is $\widehat{T}^n(x) = \widehat{T}(x) \cap G_n$.
Observe that $\widehat{T}^n(x)=\emptyset$ corresponds to the event that the tree $\widehat{T}(x)$ (rooted at $x$) is finite and of depth strictly less than $n$.
Also it is clear that $\{|\widehat{T}(x)| = \infty\} = \bigcap_n\{\widehat{T}^n(x) \neq \emptyset\}$.
For the full lattice however, we define 
\[\overline{T}^n(x) = \overline{T}_{\Lambda_n}(x) \cap G_n,\]
that is, the set of vertices in $\overline{G}$ whose lightest path to $G$ (among all paths in $\Lambda_n$) ends at $x$.
Observe that  $\overline{T}^n(x)$ doesn't have to agree with $\overline{T}(x) \cap G_n$.

\begin{lemma}\label{lemma:ergodic}
For any $1 \leq m \leq n$ and any $x \in G$, we have $\e[|T_{\Lambda_n}(x)\cap G_m|] \leq 1$.
In particular for $m = n$, we have $\e[|T^n(x)|] \leq 1$.
\end{lemma}

For the purposes of the following lemma, assume that for every finite subset $S$ of $G \times \mathbb{Z}_+$ there is an event $A_S$, such that the family $(A_S)_S$ satisfies
\begin{itemize}
	\item[(i)] $\p[A_{xS}] = \p[A_S]$, for all $x \in G$, where $x(y,n)=(xy,n)$;
	\item[(ii)] $A_{S_2} \subset A_{S_1}$, whenever $S_1 \subset S_2$;
	\item[(iii)] $A_{S_1} \cap A_{S_2} \subset A_{S_1 \cup S_2}$.
\end{itemize}

\begin{lemma}\label{lemma:ergodic_with_conditioning_general}
For all $x \in G$ and all $n>0$ and $M \geq 1$
\begin{equation}\label{eq:ergodic_with_events_directed}
\p[1 \leq |T^n(x)| \leq M, A_{T^n(x)}] \geq \p[1 \leq |T^n(x)| \leq M] - \p[A_{B_{M}(x) \times n}^c],
\end{equation}
or equivalently
\begin{equation}\label{eq:ergodic_with_events_directed-2}
\p[1 \leq |T^n(x)| \leq M, A_{T^n(x)}^c] \leq  \p[A_{B_{M}(x) \times n}^c],
\end{equation}
\end{lemma}

We will first present the proofs in the linear case $G =\mathbb{Z}$, which rely on ergodic arguments.
In the general case, the ergodic arguments are replaced by the use of the mass transport principle.

\begin{proof}[Proofs of Lemmas \ref{lemma:ergodic} and \ref{lemma:ergodic_with_conditioning_general} in the one-dimensional case]
To justify the uniform boundedness of the expectations, use ergodicity with respect to left-right translations and the fact that the sets $T^n(i)$ are connected in $\mathbb{Z}$.
In both the directed and undirected models we get for any fixed $K > 0$
\begin{equation}\label{eq:ergodic_linear_basic_1}
\e[|T^n(0)| \wedge K]  = \lim_{k \to \infty} \frac{1}{2k+1}\sum_{i=-k}^k (|T^n(i)| \wedge K).
\end{equation}
Here we need to truncate the summands at $K$, to make sure the summands are bounded and ergodic theorem applies, as the levels don't have to be bounded in the undirected case.
For the claim in Lemma \ref{lemma:ergodic_with_conditioning_general}
\begin{align}\label{eq:ergodic_linear_basic_2}
\p[0<|T^n(0)| \leq M, A_{T^n(0)}^c] & = \lim_{k \to \infty}\frac{1}{2k+1}\sum_{i=-k}^k \mathbf{1}_{\{0<|T^n(i)| \leq M, A_{T^n(i)}^c\}}.
\end{align}

For the directed case, observe that the union of level sets $\bigcup_{i=-k}^k \widehat{T}^n(i)$ is contained in the interval $[-k-n,k+n] \times \{n\}$ and thus the sum on the right hand side of \eqref{eq:ergodic_linear_basic_1} is bounded from above by $2k + 2n+1$, which implies the claim in Lemma \ref{lemma:ergodic}.
For Lemma \ref{lemma:ergodic_with_conditioning_general}, use the monotonicity of $A_S$ with respect to $S$ and observe that the sum on the right hand side of \eqref{eq:ergodic_linear_basic_2} is bounded from above by
\[
\sum_{i=-k-n}^{k+n} \mathbf{1}_{\{A_{I+(i,0)}^c\}} \leq 2n+2k+1,
\]
where $I = \{0,1,\dots,M\} \times \{n\}$.

For the undirected case, we will show that for every $\epsilon > 0$ there is a (deterministic) sequence $(k_\ell)$ such that almost surely
\begin{equation}\label{eq:ergodic_linear_undirected-to_show}
\bigcup_{-k_\ell \leq i \leq k_\ell} \overline{T}^n(i) \subset \{(i,n) : -(1+\epsilon)k_\ell \leq i \leq (1+\epsilon)k_\ell\},
\end{equation}
holds for all but finitely many $\ell$'s.
The arguments will follow as before by restricting the limits to the subsequence as $k_\ell \to \infty$ and taking the limit $\epsilon \to 0$.

To prove the claim in \eqref{eq:ergodic_linear_undirected-to_show}, fix $\epsilon > 0$.
For $i \in \mathbb{Z}$, let 
\[
\nu^n(i) = \max\{|j-i| \ | \ (j,k) \in \gamma_{\Lambda_n}(i,n) \text{, for every }0\le k\le n\}
\]
be the maximal displacement of the geodesic $\gamma_{\Lambda_n}(i,n)$ from $(i,n)$.
Since $\lim_{k \to \infty}\overline{\p}[\nu^n(0) \geq \epsilon k/2] = 0$, we can find a sequence $k_\ell\nearrow\infty$ such that $\sum_\ell \overline{\p}[\nu^n(0) \geq \epsilon k_\ell/2] < \infty$.
Setting $z_\ell=\lceil (1+\epsilon)k_\ell\rceil$
\[
\sum_\ell \Big(\overline{\p}[\nu^n(z_\ell) \geq \epsilon k_\ell/2] + \overline{\p}[\nu^n(-z_\ell) \geq \epsilon k_\ell/2]\Big) < \infty,
\]
so by Borel-Cantelli $\nu^n(z_\ell) < \epsilon k_\ell/2$ and $\nu^n(-z_\ell) < \epsilon k_\ell/2$ hold only for all but finitely many indices $\ell$.
Since the geodesics can not cross, this easily implies that \eqref{eq:ergodic_linear_undirected-to_show} holds for all but finitely many indices $\ell$.
\end{proof}

For the proof of Lemmas \ref{lemma:ergodic} and \ref{lemma:ergodic_with_conditioning_general} in the general case, the ergodic arguments will be replaced by the mass transport principle, introduced in \cite{Benjamini_et_al-99}, see also Section 8.1 in \cite{Lyonswith}.

\begin{theorem}[Mass transport principle]\label{thm:mass_transport}
Let $G $ be a countable group and let $(\Omega,\mathfs{F},\mu)$ be a probability space such that $G$ acts with measure preserving transformations on $\Omega$. 
Assume that for each $x,y \in G$ we have a non-negative random variable $m(x,y, \omega)$ which is $G$-invariant in the sense that $m(gx,gy,g\omega)  = m(x,y,\omega)$.
Then, for any $x \in G$ we have
\[
\sum_{y \in V} \e_\mu[m(x,y,\cdot)] = \sum_{y \in V} \e_\mu[m(y,x,\cdot)],
\]
where $\e_\mu$ is the expectation with respect to the measure $\mu$.
\end{theorem}

\begin{proof}[Proofs of Lemmas \ref{lemma:ergodic} and \ref{lemma:ergodic_with_conditioning_general} in the general case]
Set 
\[
m_1(y,x) = 
\left\{
\begin{array}{ll}
1, & \text{ if } (y,m) \in T_{\Lambda_n}(x) \cap G_m,\\
0, & \text{ otherwise.}
\end{array}
\right.
\]
and $m_2(y,x) = 1$ if and only if $(y,n) \in T^n(x)$, $|T^n(x)| \leq M$ and $A_{T^n(x)}^c$ hold (and $m_2(y,x) = 0$ otherwise).
Then $\sum_y m_1(y,x) = |T_{\Lambda_n}(x) \cap G_m|$ and $\sum_x m_1(y,x) = 1$, so Lemma \ref{lemma:ergodic} follows by an application of mass transport principle to $m_1$.
On the other hand
\[
\sum_y m_2(y,x) \geq  \mathbf{1}_{\{1 \leq |T^n(x)| \leq M, A_{T^n(x)}^c\}}
\]
For a fixed $y \in G$, we have that either $m_2(y,x) = 0$ for all $x$, or there is a unique $x$ such that $m_2(y,x)=1$.
In the latter case, properties (ii) and (iii) of the sets $A_S$ (property (iii) applied when $S_i$ are disjoint components of $T^n(x)$) imply that $A_{B_M(y) \times n}^c$ has to hold.
Therefore,
\[
\sum_x m_2(y,x) \leq \mathbf{1}_{A_{B_M(y) \times n}^c}.
\]
Now the claim in \eqref{eq:ergodic_with_events_directed-2} follows by applying the mass transport principle to $m_2$.
\end{proof}
By applying the same proof using the mass transport principle one can prove:
\begin{lemma}\label{lem:expwidthn}
For every $x\in G$, $\e[w_n(T(x))]=\e[|T(x)\cap G_n|]=1.$
\end{lemma}

\begin{lemma}
	\label{lemma:tree_levels}
	In both the directed and undirected case (we use setting neutral notation), if $T^n(x) \neq \emptyset$ then $T^m(x) \neq \emptyset$, for all $m \leq n$.
        Furthermore, almost surely we have that $|T(x)| = \infty$ holds if and only if $T^n(x) \neq \emptyset$ for all $n\geq 1$.
\end{lemma}

\begin{proof}
	In the directed case both claims are trivial from the definition of $\widehat{T}^n(x)$.
	For the undirected case, assume that $\overline{y} \in \overline{T}^n(x)$. 
	Then consider the geodesic $\overline{\gamma}_{\Lambda_n}(\overline{y})$ and the last point $\overline{z}$ on this geodesic which intersects the level $m$, that is $G_m$.
	Then the part of the geodesic $\overline{\gamma}_{\Lambda_n}(\overline{y})$ between $\overline{z}$ and $x$ minimizes the value $\lambda(\overline{\sigma})$ over all paths $\overline{\sigma}$ between $\overline{z}$ and $G$ which are contained in $\Lambda_n$.
	Since this path is also contained in $\Lambda_m$, it also minimizes $\lambda(\overline{\sigma})$ over all paths $\overline{\sigma}$ between $\overline{z}$ and $G$ which are contained in $\Lambda_m$, which implies $\overline{z} \in \overline{T}^m(x)$.
	Thus $\overline{T}^m(x) \neq \emptyset$.

Assuming that $|\overline{T}(x)| = \infty$ and $\overline{T}^n(x) = \emptyset$ we necessarily have $\overline{T}(x) = \overline{T}_{\Lambda_{n-1}}(x)$, which then implies that there is an $m \leq n-1$ such that $|\overline{T}_{\Lambda_{n-1}}(x) \cap G_m| = \infty$. 
This however contradicts Lemma \ref{lemma:ergodic}.

On the other hand assuming that $|\overline{T}(x)| < \infty$, we have only finitely many neighbors of $\overline{T}(x)$. 
Take $n = \max_{\overline{y}}h(\overline{\gamma}(\overline{y}))$, where the maximum is taken over all vertices $\overline{y}$ in the outer boundary of $\overline{T}(x)$.
We claim that $\overline{T}^{n+1}(x) = \emptyset$.
Assume that $\overline{T}^{n+1}(x) \neq \emptyset$ and take $\overline{z} \in \overline{T}^{n+1}(x)$. 
Since the geodesic $\overline{\gamma}_{\Lambda_{n+1}}(\overline{z})$ ends at $x$, it must contain a vertex $\overline{z}_1$ in the outer boundary of $\overline{T}(x)$.
Then the geodesic $\overline{\gamma}_{\Lambda_{n+1}}(\overline{z}_1)$ connects $\overline{z}_1$ and $x$.
However, this is impossible, since by the choice of $n$ we have $\overline{\gamma}_{\Lambda_{n+1}}(\overline{z}_1) = \overline{\gamma}(\overline{z}_1)$, and $\overline{z}_1 \notin \overline{T}(x)$.
\end{proof}

The idea of the proof of Theorem \ref{thm:main} is based on the following killing argument. 
If the tree $T(x)$ survives for a long time, with high probability we find a sequence of exceptional levels, such that with probability bounded away from zero, the tree dies out in a bounded number of levels. 
The following is a technical lemma which provides the basis for the argument.
Again we use notation oblivious to the setting.

\begin{lemma}\label{lemma:killing_argument}
Assume that there is a positive integer $N$, $\delta >0$ and a sequence of integers $n_k$ converging to infinity, such that for every $k$ there is an event $A_k \in \mathfs{F}_{n_k}$  satisfying the following
\[
A_k \subset \{T^{n_k}(x) \neq \emptyset\}, \ \ \prob[A_k] \geq \delta \prob[T^{n_k}(x) \neq \emptyset]  \ \text{ and } \  \prob[T^{n_k+N}(x) = \emptyset|A_k] \geq \delta.
\]
Then $\p[|T(x)| = \infty] = 0$.
\end{lemma}

\begin{proof}
By Lemma \ref{lemma:tree_levels} it suffices to show that the probability that $T^n(x) \neq \emptyset$ for all $n$ is equal to 0.
Assume the opposite, that $\p\Big[\bigcap_n \left\{T^n(x) \neq \emptyset\right\}\Big] =p > 0$.
By Lemma \ref{lemma:tree_levels}, the events $\left\{T^n(x) \neq \emptyset\right\}$ are decreasing, and therefore $\p[A_k] \geq \delta p$ for all $k$.
Without loss of generality we can assume that the sequence $n_k$ satisfies $n_{k+1} > n_k +N$.
We can bound the probabilities of $T^{n_k}(x) \neq \emptyset$ recursively
\begin{align*}
\p[T^{n_{k+1}}(x) \neq \emptyset] & \leq \p[T^{n_{k+N}}(x) \neq \emptyset, A_k] + \p[T^{n_{k}}(x) \neq \emptyset, A_k^c] \\
 & \leq (1-\delta)\p[A_k] + \p[T^{n_{k}}(x) \neq \emptyset] - \p[A_k] \\
 & \leq \p[T^{n_{k}}(x) \neq \emptyset] - \delta^2 p,
\end{align*}
where in the first inequality we used Lemma \ref{lemma:tree_levels}.
The above yields $\lim_{k}\p[T^{n_k}(x) \neq \emptyset] = -\infty$ which gives the contradiction.
\end{proof}

\section{One dimensional case}\label{sec:One dimensional case}

In this section we study the simplest case $G = \mathbb{Z}$ and prove Theorem \ref{thm:main}.
To reduce the notation we will assume (without loss of generality) the mean edge weights are 1, that is $\ev[\omega(e)] = 1$. In this section we denote by $\partial\BH=\BZ\times \{0\}$.

\subsection{Directed case} 
In the whole subsection we assume that random variables $\omega(e)$ satisfy the conditions in Theorem \ref{thm:main}. 

For the killing argument in the directed case we will use Lemma \ref{lemma:killing_argument}, with events $A_n$ for which there are vertices $x_1$ and $x_2$ on the $n$-th level on different sides of $\widehat{T}^n(0)$ and close to $\widehat{T}^n(0)$, such that both $\widehat{d}_\omega(x_1, \partial \widehat{\mathbb{H}})$ and $\widehat{d}_\omega(x_2, \partial \widehat{\mathbb{H}})$ are not much larger than $\min\{\widehat{d}_\omega(y, \partial \widehat{\mathbb{H}}) : y \in \widehat{T}^n(0)\}$. 
In order to achieve the lower bound on the probability $\widehat{\prob}[A_n]$ we observe that $A_n^c$ forces a geodesic in $\widehat{T}(0)$ below level $n$ not to deviate much from one of the two directions $-\theta_l$ or $-\theta_r$.  
The technical details are contained in the following lemmas.

First we present an elementary abstract result.

\begin{lemma}\label{lemma:martingale_argument}
Let $A$ be an event and $(X_n)_{n \geq 1}$ a process which  is non-decreasing on $A$. Assume that for some  $\epsilon >0$ there exist positive integers $k$ and $N$ such that for every $n \geq N$ we have $\p[A, X_n - X_{n-k} \leq (1-\epsilon)k] \leq \epsilon$. Then
 \[\p\Big[A,\limsup_{n \to \infty}\frac{X_n}{n}\geq 1-2\sqrt{\epsilon}\Big] \geq \p[A] - \sqrt{\epsilon}.\]

\end{lemma}

\begin{proof}
Choose $k$ and $N$ as in the statement and let $Y_n = \sum_{l = N}^{N+n-1} \mathbf{1}_{\{X_{lk+k}-X_{lk} \geq (1-\epsilon)k\}}$.
Denoting $p_n = \p(A, Y_n \geq (1-\sqrt{\epsilon})n)$ we have
\[
n(\p[A]-\epsilon) \leq \e[Y_n\mathbf{1}_{A}] \leq np_n + n(1-\sqrt{\epsilon})(\p[A]- p_n),
\]
which yields $p_n \geq \p[A] - \sqrt{\epsilon}$. This immediately implies
\[\p[A, \limsup_n Y_n/n \geq (1-\sqrt{\epsilon})] \geq \p[A] -\sqrt{\epsilon}\]
Since the process $X_n$ is non-decreasing, we have on $A$
  \[\limsup_n \frac{X_n}{n} \geq (1-\epsilon)\limsup_n \frac{Y_n}{n},\] which yields the claim.
\end{proof}

\begin{proposition}\label{prop:estimates_on_minima}
For $n \geq 1$ and $x = (k,n) \in \widehat{\BH}$ on the level $n$ (with $x$ possibly depending on $n$) set $W_n = \widehat{d}_\omega(x,\partial\widehat{\mathbb{H}})$.
Then there exists a constant $\kappa >0$ such that \[\lim_n \widehat{\p}[W_n < (1-\kappa)n] = 1.\]
\end{proposition}

\begin{proof}
Since $W_n$ is non-decreasing, it suffices to show the claim when taking limit along even values of $n$.
For any vertex $(k,n)$ where $-n \leq k \leq n$ for both $n$ and $k$ even, $\widehat{d}_\omega((k,n),\partial \widehat{\BH})$ can be bounded from above by the length of a shortest path from the $(k,0)$ to $(k,n)$ which never deviates more than distance 1 from the line $(k,l)$, for $0 \leq l \leq n$ (see Figure \ref{fig:direct1}).
\begin{figure}
\centering
\includegraphics[width=0.5\textwidth]{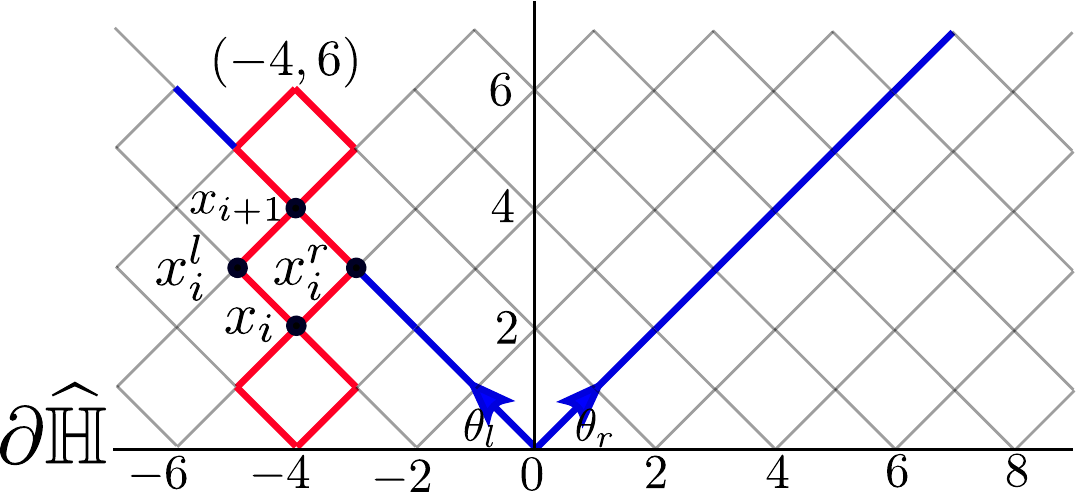}
\caption{\label{fig:direct1}}
\end{figure}
This can in turn be written as the sum of $n/2$ i.i.d.~random variables
\[
\sum_{i=0}^{n/2-1}\min\{\omega(x_{i+1},x_i^l) + \omega(x_i^l,x_{i}), \omega(x_{i+1},x_i^r) + \omega(x_i^r,x_{i})\},
\]
where $x_i = (k,2i)$, $x_i^l=x_i + \theta_l$, $x_i^r=x_i + \theta_r$.
Since random variables $\omega(e)$ are independent and have continuous distribution, the terms in the above sum have finite mean which is strictly less than $2-2\kappa$, for $\kappa$ small enough.
This proves the claim.
\end{proof}

For fixed positive integers $M$ and $k$ define the cylinders $\mathcal{C}_{M,k}^l$ and $\mathcal{C}_{M,k}^r$ as subgraphs of the directed lattice $\widehat{\BH}$ induced by the vertices 
\[\{(2i+j,j) : -M \leq 2i \leq M, 0 \leq j \leq k\}, \ \{(2i-j,j) : -M \leq 2i \leq M, 0 \leq j \leq k\},\] respectively.
We will also consider the translations of the cylinders $\mathcal{C}_{M,k}^l(x)= \mathcal{C}_{M,k}^l+ x - k \theta_r$ and $\mathcal{C}_{M,k}^r(x)= \mathcal{C}_{M,k}^r+ x - k \theta_l$, centered so that the midpoint of the upper side is at $x$.
Note that each cylinder $\mathcal{C}_{M,k}^l(x)$ and $\mathcal{C}_{M,k}^r(x)$ has exactly $M$ vertices in each level.
Top-bottom paths in these cylinders are directed paths of length $k$ going from the top side to the bottom side of the cylinder, that is $\gamma = x_0, x_1, \dots , x_k$, such that $x_i \in  \mathcal{C}_{M,k}^l(x)$ and $x_{i+1} = x_i - \theta_l$ or $x_{i+1} = x_i - \theta_r$ (and similarly for $\mathcal{C}_{M,k}^r(x)$).
See Figure \ref{fig:direct2}.
\begin{figure}
\centering
\includegraphics[width=0.45\textwidth]{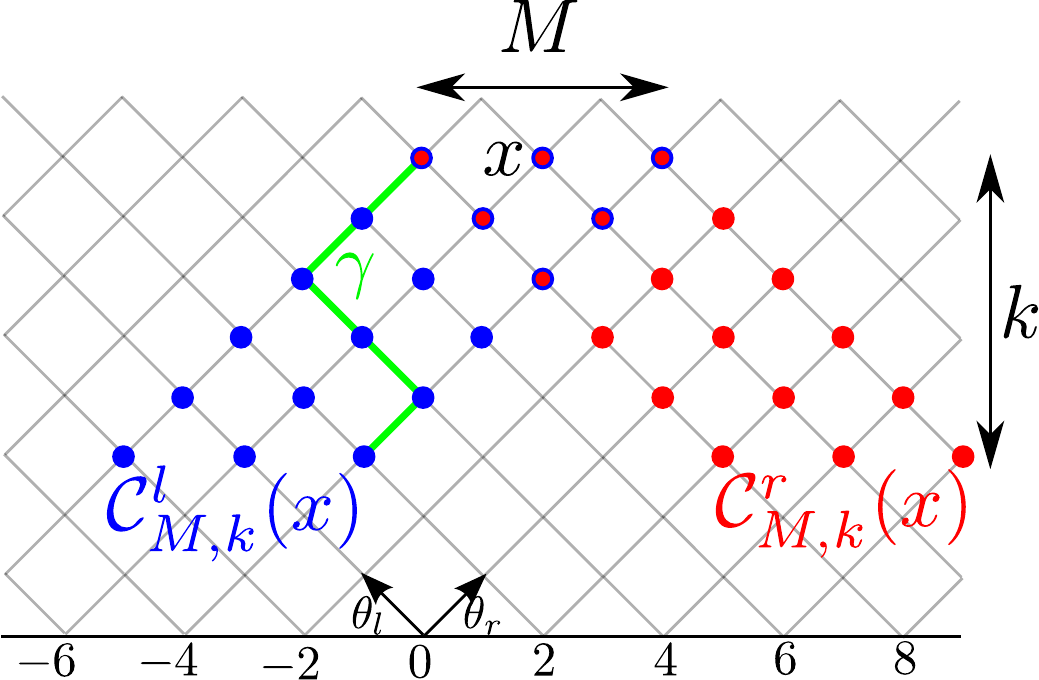}
\caption{Cylinders and Bottom-top paths.\label{fig:direct2}}
\end{figure}

\begin{lemma}\label{lemma:1dfpp}
For a fixed positive integer $M$ let $m_k^{(M)} = \min_\sigma \lambda(\sigma)$, where minimum is taken over all bottom-top paths $\sigma$ in $\mathcal{C}_{M,k}^l$ and all top-bottom paths $\sigma$ in $\mathcal{C}_{M,k}^r$.
Then for every fixed $M$, we have  $\lim_k m_k^{(M)}/k =1$, $\widehat{\prob}$ - almost surely.
\end{lemma}

\begin{proof}
The upper bound $\limsup_k m_k^{(M)}/k \leq 1$, $\widehat{\prob}$ - almost surely is obtained trivially by applying the Law of large numbers to any fixed top-bottom path.
Lower bound is proved by induction in $M$.
We focus on the cylinder $\tilde{\mathcal{C}}_{M,k}^l = \mathcal{C}_{M,k}^l - (M',0)$ shifted so that the lower right edge of the cylinder is at the origin $(0,0)$ (so $M' = M$ or $M' = M-1$, depending on the parity of $M$).
For $M=1$, the claim follows by the Law of large numbers.
For induction step, we can focus on providing the lower bounds on the passage time of the lightest top-down path from the top left vertex $(k-M',k)$, since lightest path from other vertices on the top level are completely contained in the cylinder $\tilde{\mathcal{C}}_{M-1,k}^l$, and for them the claim follows by induction hypothesis.
From $(k-M',k)$ we can consider the path $\tilde{\sigma}_k$ following the $-\theta_r$ direction and which never enters the cylinder $\tilde{\mathcal{C}}_{M-1,k}^l$ for which the lower bound is obvious, by the Law of large numbers.
Any other path has to enter the cylinder $\tilde{\mathcal{C}}_{M-1,k}^l$. 
Fix any such $\sigma$ path and assume it enters this cylinder at the level $0 \leq p < k$.
Then we have the inequality 
\[
\lambda(\sigma) \geq \lambda(\tilde{\sigma}_k) - \lambda(\tilde{\sigma}_{p}) + m_p^{(M-1)}.
\]
For a fixed $\epsilon > 0$ and $\sigma$ such that $p \leq \epsilon k$ we have
\[
\frac{\lambda(\sigma)}{k} \geq \frac{\lambda(\tilde{\sigma}_k)}{k} - \frac{\lambda(\tilde{\sigma}_{\epsilon k})}{k}.
\]
As $k \to \infty$, by the Law of large numbers the right hand side converges to $1-\epsilon$.
For $p \geq \epsilon k$, by the induction hypothesis and Law of large numbers we have for $k$ large enough, both $\lambda(\tilde{\sigma}_{p})/p$ and $m_p^{(M)}/p$ are between $1- \epsilon$ and $1 + \epsilon$.
So for $k$ large enough and any $\sigma$ such that  $p \geq \epsilon k$
\[
\frac{\lambda(\sigma)}{k} \geq \frac{\lambda(\tilde{\sigma}_k)}{k} - \frac{2\epsilon p}{k} \geq \frac{\lambda(\tilde{\sigma}_k)}{k} - 2\epsilon.
\]
The claim follows by taking $k \to \infty$, since $\epsilon > 0$ was arbitrary.

\end{proof}



Define a subgraph of $\widehat{\mathbb{H}}$ in the shape of a pentagon as follows. For a vertex $x \in \widehat{\mathbb{H}}$ and integers $M$ and $k$ ($M$ being even) let $\mathcal{P}_{M,k}(x)$ be the subgraph of $\widehat{\mathbb{H}}$ whose set of vertices is enclosed by the five sides:
\begin{align*}
  S_b &= \{x+(2i,0): -M \leq 2i \leq M\}, \text{ bottom side, }\\
  S_{lr}&=\{x+(M,0) + i\theta_r: 0 \leq i \leq k\}, \text{ lower right side, }\\
  S_{ll}&=\{x-(M,0) + i\theta_l: 0 \leq i \leq k\}, \text{ lower left side, }\\
  S_{ur}&=\{x+(M,0) + k\theta_r + i \theta_l: 0 \leq i \leq M+k\}, \text{ upper right side, } \\
  S_{ul}&= \{x-(M,0) + k\theta_l + i \theta_r: 0 \leq i \leq M+k\}, \text{ the upper right side, }
\end{align*}
and includes the vertices on the sides $S_b$, $S_{lr}$, $S_{ll}$, $S_{ur}$ and $S_{ul}$ as well.
Every edge between vertices $w$ and $y$ in $\mathcal{P}_{M,k}(x)$ will be included in the graph  $\mathcal{P}_{M,k}(x)$, if and only if at least one of $w$ and $y$ is not in  $S_b \cup S_{lr} \cup S_{ll} \cup S_{ur} \cup S_{ul}$.
Define the modified edge boundary $\tilde{\partial}\mathcal{P}_{M,k}(x)$ as the set of edges whose both endpoints are in the set $S_{lr} \cup S_{ll} \cup S_{ur} \cup S_{ul}$.
These are exactly the edges which go along the sides $S_{lr}$, $S_{ll}$,  $S_{ur}$ and $S_{ul}$.
See Figure \ref{fig:pent}.
Note that in the special case $k=0$, the pentagon $\mathcal{P}_{M,0}(x)$ collapses into a triangle.

\begin{figure}
\centering
\includegraphics[width=0.5\textwidth]{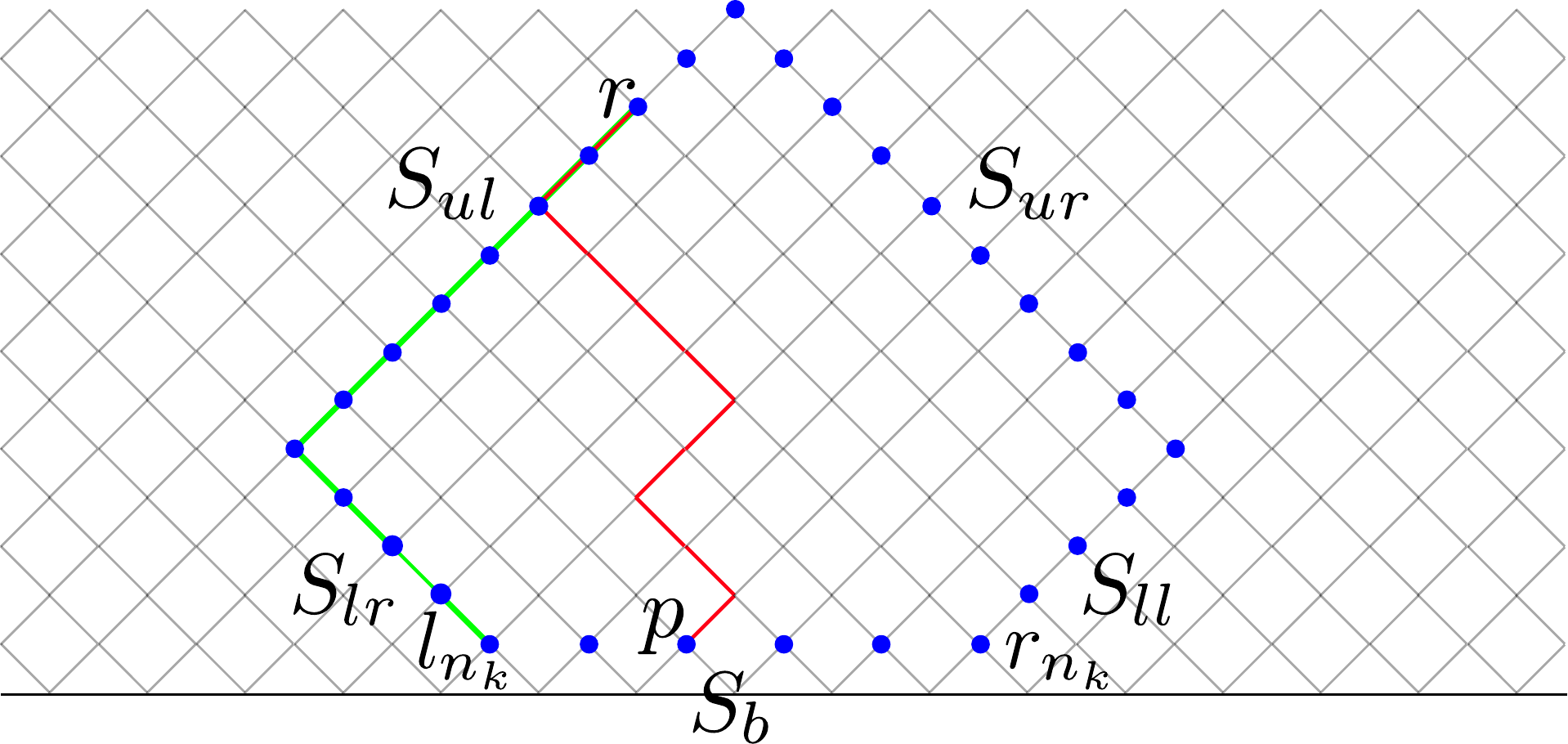}
\caption{Pentagon $\mathcal{P}_{M,k}(x)$.\label{fig:pent}}
\end{figure}

On the event $\widehat{T}^n(0) \neq \emptyset$ denote by $x_n$ the vertex in $\widehat{T}^n(0)$ which minimizes $\widehat{d}_\omega(x, \partial\widehat{\mathbb{H}})$ among all $x \in \widehat{T}^n(0)$.
For integers $n \geq k > 0$ and an even positive integer $M$, let $\mathbf{A}_{n,M,k}$ denote the event that 
\begin{itemize}
	\item[a)] $\displaystyle 1 \leq |\widehat{T}^n(0)| \leq M$ and 
	\item[b)] the geodesics $\widehat{\gamma}(x_n)$  intersects both of  the sets $\{x_n - (i-M, i): 0 \leq i \leq k\}$ and $\{x_n - (-i+M, i): 0 \leq i \leq k\}$.
\end{itemize}
The event in b) can be stated equivalently as the the event that the path consisting of the first $k$ edges of the path  $\widehat{\gamma}(x_n)$ is not completely contained in either of the cylinders $\mathcal{C}_{M-2,k}^l$ and $\mathcal{C}_{M-2,k}^r$.
Think of $k$ above as being significantly larger than $M$.
For such values of $k$, the geodesics $\widehat{\gamma}(x_n)$ will necessarily intersect at least one of the above sets, but it might fail to intersect both if the $\widehat{\gamma}(x_n)$ does not deviate much from the direction $-\theta_l$ or $-\theta_r$ for a significant amount of time.

\begin{proof}[Proof of Theorem \ref{thm:main} one dimensional directed case]
Assume that $\widehat{\p}[|\widehat{T}(0)| = \infty] >0$. First we prove the following claim. 

\textbf{Claim:} There is $\delta > 0$ such that for $M$ large enough there is $k_M$ satisfying 
\[\limsup_n \widehat{\p}[\mathbf{A}_{n,M,k}] \geq\delta, \text{ for all } k \geq k_M.\]

\textbf{Proof of the claim:}
Assume the claim is not true.
Then one can find $M$ and $\delta > 0$ for which there is a sequence $k_\ell \to \infty$ such that for all $\ell \geq 1$ we have $\widehat{\p}[\mathbf{A}_{n,M,k_\ell}]< \delta$ for $n$ large enough. 
Moreover, we can assume that $M$ and $\delta$ satisfy
\begin{equation}\label{eq:estimates_on_delta}
2\delta +M^{-1} < \min\left\{
\kappa^2/4, \widehat{\p}[|\widehat{T}(0)|=\infty]^2
\right\},
\end{equation}
where $\kappa$ is from Proposition \ref{prop:estimates_on_minima}.
Observe that the event $\mathbf{A}_{n,M,k}^c$ implies that $\widehat{\gamma}_k(x_n)$ is contained either in the cylinder $\mathcal{C}_{M,k}^l(x_n)$ or $\mathcal{C}_{M,k}^r(x_n)$.
Let $\mathbf{B}_{n,M,k}$ denote the event that 
\begin{itemize}
	\item $1 \leq |\widehat{T}^n(0)| \leq M$,
	\item there is an $x \in \widehat{T}^n(0)$ such that the minimal top-bottom path in one of the cylinders $\mathcal{C}_{M,k}^l(x)$ or $\mathcal{C}_{M,k}^r(x)$ has weight at most $(1-2\delta)k$.
\end{itemize}
We use Lemma \ref{lemma:ergodic_with_conditioning_general} in the case when the events $A_S$ are defined so that for all $x \in S$ minimal top-bottom path in one of the cylinders $\mathcal{C}_{M,k}^l(x)$ or $\mathcal{C}_{M,k}^r(x)$ has weight at least $(1-2\delta)k$.
By Lemma \ref{lemma:1dfpp}, we can find $k_0$ so that for $n \geq k_0$, we have $\widehat{\p}[\mathbf{B}_{n,M,k_0}] \leq  \delta$.
We can choose $k_0$ so that $\widehat{\p}[\mathbf{A}_{n,M,k_0}]< \delta$ for $n$ large enough. 
Observe that $\widehat{d}_\omega(x_{n}, \partial \widehat{\mathbb{H}}) - \widehat{d}_\omega(x_{n-k_0}, \partial \widehat{\mathbb{H}}) \leq (1-2\delta)k_0$ implies that $\lambda(\widehat{\gamma}_{k_0}(x_n)) \leq (1-2\delta)k_0$.
Therefore, we have the inclusion
\[
\{1 \leq |\widehat{T}^n(0)|  \leq M, \widehat{d}_\omega(x_{n}, \partial \widehat{\mathbb{H}}) - \widehat{d}_\omega(x_{n-k_0}, \partial \widehat{\mathbb{H}}) \leq (1-2\delta)k_0 \} \subset \mathbf{A}_{n,M,k_0} \cup \mathbf{B}_{n,M,k_0}.
\]
By the bounds on the probability of $\mathbf{A}_{n,M,k_0}$ and $\mathbf{B}_{n,M,k_0}$, and Lemma \ref{lemma:ergodic} we now have
\[\widehat{\p}\left[\widehat{T}^n(0) \neq \emptyset, \widehat{d}_\omega(x_{n}, \partial \widehat{\mathbb{H}}) - \widehat{d}_\omega(x_{n-k_0}, \partial \widehat{\mathbb{H}}) \leq (1-2\delta)k_0\right] \leq 2\delta + M^{-1}.\]
Now applying Lemma \ref{lemma:martingale_argument} with $A = \{|\widehat{T}(0)|= \infty\}$ and $X_n = \widehat{d}_\omega(x_n, \partial\widehat{\mathbb{H}})\mathbf{1}_{\{\widehat{T}^n(0) \neq \emptyset\}}$ yields that
\[
\widehat{\p}\Big[|\widehat{T}(0)| = \infty, \limsup_n \frac{\widehat{d}_\omega(x_n, \partial\widehat{\mathbb{H}})}{n} \geq 1-2\sqrt{2\delta+M^{-1}}\Big] = \widehat{\p}[|\widehat{T}(0)| = \infty]-\sqrt{2\delta+M^{-1}} >0.
\]
Since $1-2\sqrt{2\delta+M^{-1}} > 1-\kappa$, this gives a contradiction with Proposition \ref{prop:estimates_on_minima}.

For the basis of the construction of the events $A_k$ in Lemma \ref{lemma:killing_argument} take an even $M$ such that for any $n$ \[\widehat{\p}[|\widehat{T}^n(0)|\geq M] \leq \frac{1}{M} \leq \frac{\delta\widehat{\p}[\widehat{T}^n(0) \neq \emptyset]}{6},\]
(guaranteed by Lemma \ref{lemma:ergodic}), where $\delta$ is as above. 
For this particular value of $M$, find $k_0$ and a subsequence $n_\ell$ such that  $\widehat{\p}[\mathbf{A}_{n_\ell,M,k_0}] \geq \delta$. For these values of $M$ and $k_0$ choose $0 < \xi < \esssup \omega(e)$ such that the event $\mathbf{C}_{n,M,k_0}$ defined as 
\begin{itemize}
\item $1 \leq |\widehat{T}^n(0)| \leq M$, 
\item $\omega(e) <\xi$ for all $e \in \mathcal{C}_{M,k_0}^l(x) \cup \mathcal{C}_{M,k_0}^r(x)$ and all  $x \in \widehat{T}^n(0)$
\end{itemize}
has probability at least
\[
\widehat{\p}[1 \leq |\widehat{T}^n(0)| \leq M] - \frac{\delta}{3}\widehat{\p}[\widehat{T}^n(0) \neq \emptyset],
\] 
for every $n> k_0$. 
The existence of such $\xi$ is guaranteed by Lemma \ref{lemma:ergodic_with_conditioning_general} and since $\widehat{\p}[|\widehat{T}(0)| = \infty]>0$ gives a lower bound on the probabilities $\widehat{\p}[\widehat{T}^n(0) \neq \emptyset]$.
Finally define the event
\[
A_\ell =  \mathbf{A}_{n_\ell,M,k_0} \cap \mathbf{C}_{n_\ell,M,k_0} \in\mathfs{F}_{n_\ell}.
\]
Since both $A_{n_\ell,M,k_0}$ and $C_{n_\ell,M,k_0}$ are contained in $\{\widehat{T}^{n_\ell}(0) \neq \emptyset\}$, the union bound gives
$\widehat{\p}[A_\ell] \geq \delta \widehat{\p}[\widehat{T}^{n_\ell}(0) \neq \emptyset]/3$.
To apply Lemma \ref{lemma:killing_argument} and finish the proof we only need to observe that for an appropriately chosen $N$ the probabilities $\widehat{\p}[\widehat{T}^{n_\ell+N}(0) = \emptyset|A_\ell]$ are bounded away from zero uniformly in $\ell$. The rest of the proof is devoted to this.

First define the points $y^l_{n_\ell}$ and $y^r_{n_\ell}$ as $y^l_{n_\ell} = x_{n_\ell} - (M,0)$ and $y^r_{n_\ell} = x_{n_\ell} + (M,0)$ (recall that $M$ is chosen to be even).
Observe that on the event $A_\ell$ we have $y^l_{n_\ell} \notin \widehat{T}(0)$ and $y^r_{n_\ell} \notin \widehat{T}(0)$.
Furthermore, on the event $A_\ell$ the path $\widehat{\gamma}_{k_0}(x_{n_\ell})$ intersects the sides of the cylinders $\{y^l_{n_\ell} - i\theta_l : 0 \leq i \leq k_0\}$ and $\{y^r_{n_\ell} - i\theta_r : 0 \leq i \leq k_0\}$, so choose points $q_l \in \widehat{\gamma}_{k_0}(x_{n_\ell}) \cap \{y^l_{n_\ell} - i\theta_l : 0 \leq i \leq k_0\}$ and $q_r \in \widehat{\gamma}_{k_0}(x_{n_\ell}) \cap \{y^r_{n_\ell} - i\theta_r : 0 \leq i \leq k_0\}$.
By the definition of the event $A_\ell$ we have that both $\widehat{d}_\omega(y^l_{n_\ell}, \partial \widehat{\mathbb{H}})$ and $\widehat{d}_\omega(x_{n_\ell}, \partial \widehat{\mathbb{H}})$ are between $\widehat{d}_\omega(q_{l}, \partial \widehat{\mathbb{H}})$ and $\widehat{d}_\omega(q_{l}, \partial \widehat{\mathbb{H}}) + \xi k_0$ and so by symmetry
\begin{equation}\label{eq:estimates_before_killing}
|\widehat{d}_\omega(y^l_{n_\ell}, \partial \widehat{\mathbb{H}}) - \widehat{d}_\omega(x_{n_\ell}, \partial \widehat{\mathbb{H}})| \leq 2\xi k_0 \text{ and } |\widehat{d}_\omega(y^r_{n_\ell}, \partial \widehat{\mathbb{H}}) - \widehat{d}_\omega(x_{n_\ell}, \partial \widehat{\mathbb{H}})| \leq 2\xi k_0.
\end{equation}
Next take $\xi_1 < \xi_2$ such that both probabilities $\widehat{\p}(\omega(e) > \xi_2)$ and $\widehat{\p}(\omega(e) < \xi_1)$ are positive, and take a positive integer $k_1$ such that $k_1 > 2\xi k_0/(\xi_2-\xi_1)$.

Consider the pentagon $\mathcal{P}_{M,k_1}(x_{n_\ell})$ and denote its lower left, lower right upper left and upper right sides with $S_{ll}$, $S_{lr}$, $S_{ul}$, $S_{ur}$.
Consider the event
\[
\mathbf{D}_{M,k_1}(x_{n_\ell})  = \{\omega(e) > \xi_2 : e \in \mathcal{P}_{M,k_1}(x_{n_\ell})\} \bigcap \{\omega(e) < \xi_1 : e \in \tilde{\partial}\mathcal{P}_{M,k_1}(x_{n_\ell})\},
\]
that is we require all the edges in $\mathcal{P}_{M,k_1}(x_{n_\ell})$ to be heavier than $\xi_2$ and all the edges on the lower left, lower right, upper left and upper right sides of $\mathcal{P}_{M,k_1}(x_{n_\ell})$ to be lighter than $\xi_1$.
Obviously, for fixed values of $M,k_1,\xi_1,\xi_2$, on the event $A_\ell$ the conditional probability $\widehat{\p}[\mathbf{D}_{M,k_1}(x_{n_\ell})|\mathfs{F}_{n_\ell}]$ is bounded away from zero, uniformly in $\ell$.
We show that given $A_\ell$, on the event $\mathbf{D}_{M,k_1}(x_{n_\ell})$ we have $(S_{ul}\cup S_{ur}) \cap \widehat{T}(0) = \emptyset$.
Then $\widehat{T}^{n_\ell+N}(0) = \emptyset$ for $N= k_1+M$, since otherwise for any $y \in \widehat{T}^{n_\ell+N}(0)$, the path $\widehat{\gamma}(y)$ intersects $\widehat{T}^{n_\ell}(0)$, and in particular either $S_{ul}$ or $S_{ur}$.
Thus for $N= k_1+M$ we get the lower bound 
\[
\widehat{\p}[\widehat{T}^{n_\ell+N}(0) = \emptyset|A_\ell] \geq \widehat{\p}[\mathbf{D}_{M,k_1}(x_{n_\ell})|A_\ell] >0,
\]
which is uniform in $\ell$.
The claim then follows by Lemma \ref{lemma:killing_argument}.

Assume the contrary, that for some $z \in S_{ul}$ we have $\{z \in \widehat{T}(0)\} \cap A_\ell \cap \mathbf{D}_{M,k_1}(x_{n_\ell}) \neq \emptyset$.
On the intersection of these events denote $p = \widehat{T}^{n_k}(0)\cap \widehat{\gamma}(z)$. The part of the geodesics $\widehat{\gamma}(z)$ between the points $p$ to $z$  contains at least $k_1$ edges from  $\mathcal{P}_{M,k_1}(x_{n_\ell})$, the other edges might be a part of the side $S_{ul}$.
Considering the path from $y^l_{n_\ell}$ to $z$ following the edges of $S_{ll}$ and $S_{ul}$ it is an easy observation that on the event $\mathbf{D}_{M,k_1}(x_{n_\ell})$ we have
\[\widehat{d}_\omega(y^l_{n_\ell},z) \leq \widehat{d}_\omega(p,z) - k_1(\xi_2-\xi_1).\]
Considering an even $i \neq 0$ such that $y^l_{n_\ell} \in \widehat{T}(i)$ observe that
\begin{multline*}
\widehat{d}_\omega(z,(i,0)) \leq \widehat{d}_\omega(z,y^l_{n_\ell}) + \widehat{d}_\omega(y^l_{n_\ell},(i,0)) \leq \widehat{d}_\omega(p,z) -k_1(\xi_2-\xi_1) + \widehat{d}_\omega(y^l_{n_\ell},\partial\widehat{\mathbb{H}}) \\
\leq \widehat{d}_\omega(p,z) -k_1(\xi_2-\xi_1) + \widehat{d}_\omega(x_{n_\ell},\partial\widehat{\mathbb{H}}) + 2\xi k_0 < \widehat{d}_\omega(p,z) + \widehat{d}_\omega(p,\partial\widehat{\mathbb{H}})  = \widehat{d}_\omega(z,0).
\end{multline*}
This gives the contradiction. The fact that $S_{ur} \cap \widehat{T}(0) = \emptyset$ follows by symmetry.

\end{proof}

\subsection{Undirected case}

Let $n$ be a positive integer and $I$ a subset of consecutive vertices on the level $n$, that is $I = \{(i_l,n), (i_l+1,n), \dots (i_r,n)\}$.
Let  $\mathcal{R}_{I,k}$ be the rectangle with base $I \cup \{(i_l-1,n),(i_r+1,n)\}$ and of height $k$, and $S_{b}$, $S_l$, $S_r$ and $S_u$ the bottom, the left, the right and the upper side of $\mathcal{R}_{I,k}$.
More precisely, define $S_{b}$, $S_l$, $S_r$ and $S_u$ to be subgraphs with the sets of vertices
\begin{align*}
& I \cup \{(i_l-1,n),(i_r+1,n)\}, \ \{(i_l-1,j) : n \leq j \leq n+k-1\}, \\ & \{(i_r+1,j) : n \leq j \leq n+k-1\}, \ \{(i,n+k-1) : i_l-1 \leq i \leq i_r+1\},
\end{align*}
respectively.
Set $S_{b}$, $S_l$, $S_r$ and $S_u$ to be the subgraphs induced by their respective sets of vertices, that is they contain all edges between any two of their vertices.
Now define  $\mathcal{R}_{I,k}$ as a subgraph with the set of vertices \[\{(i,j): i_l-1 \leq i \leq i_r+1, n \leq j \leq n+k-1\},\] and include in  $\mathcal{R}_{I,k}$ all edges  $e=(x,y)$ between two vertices $x$ and $y$ of $\mathcal{R}_{I,k}$ such that $e \notin S_{b} \cup S_l \cup S_r \cup S_u$.
Define a modified boundary of $\mathcal{R}_{I,k}$ to be the union of subgraphs $\tilde{\partial} \mathcal{R}_{I,k} = S_l \cup S_r \cup S_u$, that is we don't include the bottom side in the boundary.

\begin{proof}[Proof of Theorem \ref{thm:main} one dimensional undirected case]
Again we apply Lemma \ref{lemma:killing_argument}.
Assume that $\overline{\p}[|\overline{T}(0)|=\infty] >0$. 
By Lemma \ref{lemma:tree_levels} we have that \[\overline{\p}\Big[\bigcap_n\{\overline{T}^n(0) \neq \emptyset\}\Big] = \overline{\p}[|\overline{T}(0)|=\infty] >0.\]
By Lemma \ref{lemma:ergodic} for $M$ large enough and any $n$ we have $\overline{\p}[0 < |\overline{T}^n(0)| \leq M] \geq \overline{\p}[|\overline{T}(0)|=\infty]/2$.
On the event $\overline{T}^n(0) \neq \emptyset$, denote the vertices $\overline{T}^n(0)$ by $(j,n)$ for $j_l \leq j \leq j_r$ and define $y^l_n = (j_l-1,n)$ and $y^r_n= (j_r+1,n)$.
Denote $i_l, i_r \in \partial \overline{\BH}$ such that $y^l_n \in \overline{T}^n(i_l)$ and $y^r_n \in \overline{T}^n(i_r)$.
By definition, $i_l \neq 0$ and $i_r \neq 0$, however note that without further assumptions we can not claim that $y^l_n$ and $y^r_n$ are not in $\overline{T}(0)$.
Now by Lemma \ref{lemma:ergodic_with_conditioning_general} we can find positive real numbers $\xi$ and $\delta$ such that for every $n$ the event $A_n$ defined as
\begin{itemize}
	\item $0 < |\overline{T}^n(0)| \leq M$,
	\item $\omega(e) < \xi$, for all horizontal edges  $e$  with at least one endpoint in $\overline{T}^n(0))$,
\end{itemize}
has probability at least $\delta$.
Without loss of generality we can assume that there are numbers $\xi_1$ and $\xi_2$ such that $\xi_1< \xi < \xi_2$ and such that both probabilities $\overline{\p}(\omega(e)<\xi_1)$ and $\overline{\p}(\omega(e)> \xi_2)$ are positive.
We will show that for an appropriate choice of  $N$, the probabilities $\overline{\p}[\overline{T}^{n+N}(0) = \emptyset|A_n]$ are uniformly bounded away from zero, which  by Lemma \ref{lemma:killing_argument} proves the claim.
Fix an integer $N$ with the property that
\[
N > \frac{(\xi + \xi_1)M + 3\xi}{\xi_2 - \xi_1}.
\]
Consider the rectangle $\mathcal{R}_{\overline{T}^n(0),N}$, and the event
\[
R_{n,N} = \{\omega(e) < \xi_1, \text{ for all } e \in \tilde{\partial} \mathcal{R}_{\overline{T}^n(0),N}\} \cap \{\omega(e) > \xi_2, \text{ for all } e \in \mathcal{R}_{\overline{T}^n(0),L}\}.
\]
Given the event $A_n$ the width of $\overline{T}^n(0)$ is bounded by $M+2$, and so the event $R_{n,N}$ puts constraints on weights of less than $2(M+2)N$ edges.
Thus $\overline{\p}[R_{n,N}|A_n] > \delta$, for some $\delta>0$ and all positive integers $n$.
Next we prove that on the event $R_{n,N} \cap A_n$ we necessarily have 
$\overline{T}^{n+N}(0) = \emptyset$. By Lemma \ref{lemma:killing_argument} this will finish the proof.

Assume that there is a vertex $x \in  \overline{T}^{n+N}(0)$. 
Denote the vertices in geodesics $\overline{\gamma}_{\Lambda_{n+N}}(x)$ by $\overline{\gamma}_{\Lambda_{n+N}}(x) = x_0, x_1, \dots , x_k$, and $x_i = (m_i,j_i)$, so that $x_0 = x$, $|x_{i-1} - x_i| =1$ and $x_k = (0,0)$.
Take index $i_1$ so that $x_{i_1}$ is on level $n$, $x_{i_1-1}$ is on level $n+1$, and $x_i$ does not go above level $n$ for $i > i_1$. 
More precisely, $j_{i_1-1} = n+1$, $j_{i_1} =n$ and $j_i \leq n$ for all $i \geq i_1$.
Then it is a simple observation that $x_{i_1} \in \overline{T}^n(0)$.
Assume that there is an index $i_2 < i_1$ such that $x_{i_2}$ is also on level $n$, that is $j_{i_2} = n$. 
Take the largest such index $i_2$, that is $j_i > n$ for all $i_2 < i < i_1$.
Then observe that we have one of two possibilities:
\begin{itemize}
	\item  either $x_{i_2} \in \overline{T}^n(0)$ and all edges in the part of $\overline{\gamma}_{\Lambda_{n+N}}(x)$ between the vertices $x_{i_2}$ and $x_{i_1}$ are in $\mathcal{R}_{\overline{T}^n(0),N}$, or
	\item  for some index $i_0$ such that $i_2 < i _0< i_1$ the point $x_{i_0}$ is on the boundary $\tilde{\partial}\mathcal{R}_{\overline{T}^n(0),N}$.
\end{itemize}
The first scenario is impossible, since  the part of the geodesic $\overline{\gamma}_{\Lambda_{n+N}}(x)$ between the points $x_{i_2}$ and $x_{i_1}$ would have the weight at least $\xi_2 |i_1-i_2|$, while connecting the points $x_{i_2}$ and $x_{i_1}$ with the horizontal line (with all the edges on the $n$-th level) has the smaller weight of at most $\xi |i_1-i_2|$.
Therefore, we know the second scenario holds, and take the largest index $i_0$ such that $i_2 < i_0 < i_1$ and  $x_{i_0} \in \tilde{\partial}\mathcal{R}_{\overline{T}^n(0),N}$.
By the choice of $i_0$ it is clear that all the vertices and edges in $\overline{\gamma}_{\Lambda_{n+N}}(x)$ between $x_{i_0}$ and $x_{i_1}$ are contained in the rectangle $\mathcal{R}_{\overline{T}^n(0),N}$.
Denote the points $z_1 = x_{i_0}$ and $z_2 = x_{i_1}$.


Observe that the part of $\overline{\gamma}_{\Lambda_{n+N}}(x)$ appearing after $z_2$ coincides with the geodesic $\overline{\gamma}_{\Lambda_{n+N}}(z_2) = \overline{\gamma}_{\Lambda_{n}}(z_2)$.
In particular $\overline{d}_{\omega,\Lambda_{n+N}}(z_2,\partial \overline{\BH}) = \overline{d}_{\omega,\Lambda_{n}}(z_2,\partial \overline{\BH})$.
Connecting $y^l_n$ to $z_2$ by the shortest horizontal path and then using $\overline{\gamma}_{\Lambda_{n}}(z_2)$ to connect to $\partial \overline{\BH}$ yields
\begin{equation}\label{eq:comparing_end_paths}
\overline{d}_{\omega,\Lambda_{n}}(y^l_n,i_l)  < \xi|y^l_n-z_2| + \overline{d}_{\omega,\Lambda_{n}}(z_2, \partial \overline{\BH}).
\end{equation}

Assuming that $z_1$ is on the left side of the the rectangle, $z_1 \in S_l$,  observe that the part of the geodesic $\overline{\gamma}_{\Lambda_{n+N}}(x)$ between $z_1$ and $z_2$ has at least $|z_1-y^l_n| + |y^l_n-z_2|$ edges in $\mathcal{R}_{\overline{T}^n(0),N}$, so on the event $R_{n,N} \cap A_n$, the weight of this path is at least $\xi_2(|z_1-y^l_n| + |y^l_n-z_2|)$. Therefore,
\begin{equation}\label{eq:connecting_through_the_last_point_on_level_n_left}
\overline{d}_{\omega,\Lambda_{n+N}}(z_1,\partial \overline{\BH}) \geq \xi_2(|z_1-y^l_n| + |y^l_n-z_2|) + \overline{d}_{\omega,\Lambda_n}(z_2, \partial \overline{\BH}).
\end{equation}
On the other hand, the shortest path connecting $z_1$ to $y^l_n$ has weight at most $\xi_1|z_1-y^l_n|$. Then traversing the shortest path connecting $y^l_n$ to $i_l$ which stays below level $n$ and making use of \eqref{eq:comparing_end_paths} gives
\begin{align*}
\overline{d}_{\omega,\Lambda_{n+N}}(z_1,i_l) & \leq \xi_1|z_1-y^l_n| + \overline{d}_{\omega,\Lambda_n}(y^l_n,i_l) \\
& < \xi_1|z_1-y^l_n| + \xi|y^l_n-z_2| +\overline{d}_{\omega,\Lambda_n}(z_1, \partial \overline{\BH}).
\end{align*}
This gives the contradiction with \eqref{eq:connecting_through_the_last_point_on_level_n_left}. The case $z_1 \in S_r$ is handled in the same way by replacing the role of $y^l_n$ with $y^r_n$.

The case $z_1 \in S_u$, when $z_1$ is on the upper side of the rectangle is handled analogously.
Observe that now the part of the geodesics $\overline{\gamma}_{\Lambda_{n+N}}(x)$ connecting $z_1$ with $z_2$ has at least $N$ edges in $\mathcal{R}_{\overline{T}^n(0),N}$ so \eqref{eq:connecting_through_the_last_point_on_level_n_left} is replaced by $\overline{d}_{\omega,\Lambda_{n+N}}(z_1,\partial \overline{\BH}) \geq \xi_2 N + \overline{d}_{\omega,\Lambda_n}(z_2, \partial \overline{\BH})$. On the other hand, the shortest path connecting $z_1$ to $y^l_n$ with the edges in $\tilde{\partial} \mathcal{R}_{\overline{T}^n(0),N}$ which run along the upper and then along the left side of the rectangle has weight at most $\xi_1 (M+N+2)$. As in the previous case this yields
\begin{align*}
\overline{d}_{\omega,\Lambda_{n+N}}(z_1,i_l) & \leq \xi_1(M+N+2) + \overline{d}_{\omega}(y^l_n,i_l)\\
& < \xi_1(M+N+2) + \xi(M+1) + \overline{d}_{\omega,\Lambda_n}(z_2, \partial \overline{\BH}).
\end{align*}
Thus we obtain $\xi_1(M+N+2) + \xi(M+1) \geq \xi_2 N$, which is false by the assumption on $N$.
This gives the contradiction and finishes the proof.
\end{proof}

\section{General base graphs}\label{sec:general}

In this section we prove Theorem \ref{thm:main} in the general case, where $G$ is a Cayley graph of a finitely generated countable group. Recall that $d(\cdot,\cdot)$ is the graph metric in $G$ and $B_R(x)=\{y \in G : d(x,y) \le R\}$. 
We will use the notation $\mathring{B}_R(x)$ for the open ball in $G$ of radius $R$  around $x$, that is $\mathring{B}_R(x) = \{y \in G : d(x,y) < R\}$. 
By $S_R(x)$ we will also denote the sphere in $G$ of radius $R$ around $x$, that is $S_R(x) = \{y \in G : d(x,y) = R\}$. 
Also, we will denote the projection of elements of $G \times \mathbb{Z}^+$ onto $G$ by  $\mathfs{P}(x,n) = x$.

\subsection{Directed case for general graphs}

Observe that any path in $G$ of length $k$ between $x$ and $y$ can be lifted to a directed path from $(x,n+k)$ to $(y,n)$ in $\widehat{G}$.
In particular, any closed path in $G$ of length $k$ containing a vertex $x \in G$, can be lifted to a path between $(x,n+k)$ and $(x,n)$.
In $G$ there is certainly a closed path of length $k$ for any even $k$, so for even $k$ there is a directed path between $(x,n+k)$ and $(x,n)$.
If $G$ is bipartite, then actually $(x,n)$ and $(x,n+k)$ are different components of $\widehat{G}$ for odd $k$.
For a non-bipartite graph $G$, let $m$ denote the length of the shortest closed path of odd length. 
Then any $k \geq 2m-1$ can be written as a sum of a non-negative multiple of $m$ and a non-negative multiple of 2, and thus for any $k \geq 2m-1$ there is a closed path in $G$ of length $k$.
Thus if $G$ is non-bipartite, for $k$ large enough there is a directed path between $(x,n+k)$ and $(x,n)$ for any $x$ and $n$.
The smallest such $k$ we denote by $\mu_G$.

For a fixed vertex $x$ of $G$ let $m_n = \widehat{\e}[\widehat{d}_\omega((x,n),G)]$ denote the expected passage time from the vertex $(x,n)$ to the base graph $G$.
By stationarity, $m_n$ does not depend on the choice of $x \in G$.



\begin{lemma}\label{lemma:vertical_growth}
Let $k$ be a positive integer, which we assume to be even if $G$ is bipartite.
For any $\epsilon > 0$ there exists an  $K>0$ (depending on $\epsilon$ and $k$) such that for any $x \in G$ and any positive integer $n$ we have
\[
\widehat{\p}[ \widehat{d}_\omega((x,n+k),G) - \widehat{d}_\omega((x,n),G) < -K ] \leq \epsilon.
\]
\end{lemma}

\begin{proof}
First observe that $m_n$ is an increasing sequence in $n$.
Recalling $G_k = G\times \{k\}$, this follows from a rather obvious inequality $\widehat{d}_\omega((x,n+k),G_k) < \widehat{d}_\omega((x,n+k),G)$, since then
\[m_n = \widehat{\e}[\widehat{d}_\omega((x,n+k),G_k)] < \widehat{\e}[\widehat{d}_\omega((x,n+k),G)] = m_{n+k}.\]

Denote the event $\{\widehat{d}_\omega((x,n+k),G) - \widehat{d}_\omega((x,n),G) < -K\}$ from the statement by $A_{n,k;K}$.
Now assume that there is a directed path between $(x,n+k)$ and $(x,n)$.
Take such a directed path $\widehat{\sigma}$ from $(x,n+k)$ to $(x,n)$, and let $\lambda(\widehat{\sigma}) = \sum_{\widehat{e} \in \widehat{\sigma}}\omega(\widehat{e})$ be the sum of weights of all edges in $\widehat{\sigma}$ and observe that
\[
\widehat{d}_\omega((x,n+k),G) - \widehat{d}_\omega((x,n),G) \leq \lambda(\widehat{\sigma}).
\]
One can bound $\widehat{d}_\omega((x,n+k),G) - \widehat{d}_\omega((x,n),G)$ from above by $-K$ on the event $A_{n,k;K}$ and by $\lambda(\widehat{\sigma})$ on  $A_{n,k;K}^c$. 
Applying the expectation to this inequality and using the fact that $\lambda(\widehat{\sigma})$ is always positive, we get
\[
0 \leq m_{n+k} - m_n = \widehat{\e}[\widehat{d}_\omega((x,n+k),G) - \widehat{d}_\omega((x,n),G)] \leq -K\widehat{\p}[A_{n,k;K}] + \widehat{\e}[\lambda(\widehat{\sigma})],
\]
which gives $\widehat{\p}[A_{n,k;K}] \leq \widehat{\e}[\lambda(\widehat{\sigma})]/K$.
Since $\widehat{\e}[\lambda(\widehat{\sigma})]$ does not depend on $n$ or $K$, this completes the proof when $G$ is bipartite, and for all $k \geq \mu_G$, when $G$ is not bipartite.
When $G$ is not bipartite and $k < \mu_G$, observe that the event $\{\widehat{d}_\omega((x,n+k),G) - \widehat{d}_\omega((x,n),G) \leq -K\}$ is contained in the union of the events 
\[\{\widehat{d}_\omega((x,n+k+\mu_G),G) - \widehat{d}_\omega((x,n),G) \leq -K/2\}\]
and 
\[\{\widehat{d}_\omega((x,n+k+\mu_G),G) - \widehat{d}_\omega((x,n+k),G) \geq K/2\}.\]
The probability of the first event is less than $\epsilon/2$ for $K$ large enough by the proven part of the lemma.
To bound the probability of the second event by $\epsilon/2$, use the fact that $(x,n+k)$ and $(x,n+k+ \mu_G)$ can be connected by a path $\widehat{\sigma}'$ of length $\mu_G$, and so, similarly as above, the second event implies that $\lambda(\widehat{\sigma}') \geq K/2$.
Now the desired bound for the second event follows by Markov inequality for $K$ large enough.
\end{proof}



If $G$ is a non-bipartite graph, let $B_R'(x)=B_R(x)$ denote the graph-metric ball of radius $R$ around $x \in G$, and if $G$ is a bipartite graph require additionally that $d(x,y)$ is even, that is $B_R'(x) = \{y \in G : d(x,y) \leq R, \text{ and } d(x,y) \text{ is even}\}$.
For a fixed vertex $\widehat{x} = (x,n)$ set $D_{\widehat{x},R;K}$ to be the event for which 
\begin{equation}\label{eq:bad_vertices}
|\widehat{d}_\omega((y,n),G) - \widehat{d}_\omega(\widehat{x},G)| \leq K,
\end{equation}
for all $y \in B_R'(x)$.

\begin{lemma}\label{lemma:not_too_light_neighborhood}
For any $R>0$ and $\epsilon >0$, there exists $K>0$ such that $\widehat{\p}[D_{\widehat{x},R;K}] \geq 1- \epsilon$, for any $\widehat{x} \in \widehat{G}$.
\end{lemma}

\begin{proof}
Fix $\widehat{x} = (x,n)$.
Observe that there exists a positive integer $k$, which depends only on $R$,  such that for each $y \in B_R'(x)$ there is a directed path $\widehat{\sigma}_y^\downarrow$ path from $(y,n)$ to $(x,n-k)$ (assuming that $n \geq k$).
If $G$ is bipartite simply take $k \geq R$ to be even, and if $G$ is not bipartite take $k=R+ \mu_G$.
Mirroring the path $\widehat{\sigma}_y^\downarrow$ above the level $n$, we obtain a directed path  $\widehat{\sigma}_y^\uparrow$ from $(x,n+k)$ to $(y,n)$.
Since the size of $B_R'(x)$  is bounded and depends only on $R$, by a union bound we can find $K_1>0$ such that with probability at least $1-\epsilon/2$ we have both  $\lambda(\widehat{\sigma}_y^\downarrow) \leq K_1$ and $\lambda(\widehat{\sigma}_y^\uparrow) \leq K_1$, for all $y \in B_R'(x)$.
Next by Lemma \ref{lemma:vertical_growth} we can find $K_2>0$ such that with probability at least $1-\epsilon/2$ we have both  $\widehat{d}_\omega((x,n),G) \geq \widehat{d}_\omega((x,n-k),G)-K_2$ and $\widehat{d}_\omega((x,n+k),G) \geq \widehat{d}_\omega((x,n),G)-K_2$.
Since 
\[\widehat{d}_\omega((x,n+k),G) - \lambda(\widehat{\sigma}_y^\uparrow) \leq \widehat{d}_\omega((y,n),G) \leq \widehat{d}_\omega((x,n-k),G) + \lambda(\widehat{\sigma}_y^\downarrow),\]  
the intersection of the above two events implies that 
\[
|\widehat{d}_\omega((y,n),G)- \widehat{d}_\omega(\widehat{x},G)| \leq K_1 + K_2,
\]
for all $y \in B'_R(x)$.
This yields the claim for $K= K_1+K_2$, and for $n \geq k$.
Since the value of $k$ depends only on the graph $G$ and $R$, the cases $n < k$ are handled by increasing $K$ if necessary.
\end{proof}

\begin{remark}\label{rem:dead_ends}
It is well known that in general Cayley graphs, one can find paths starting at a fixed vertex, which can not be extended (e.g.~Lamplighter graphs).
In other words there are vertices $x$ and $y$ such that every neighbor $z$ of $y$ satisfies $d(x,z) \leq d(x,y)$,
and therefore, for $y \in \mathring{B}_R(x)$, we might have $d(y,S_R(x)) > R - d(x,y)$.
However, note that by backtracking from $y$ to $x$ and then moving to $S_R(x)$, there is path from $y$ to $S_R(x)$ of length $R + d(x,y) < 2R$.
In particular, for every $y \in \mathring{B}_R(x)$ and every $k \geq 2R$ of the same parity as $d(y,S_R(x))$, there is a path from $y$ to a point on $S_R(x)$.
To handle the parity issue, observe that for every $y \in \mathring{B}_R(x)$ and every $k \geq 2R$, there is a path from $y$ to a point on $S_R(x) \cup S_{R+1}(x)$.
Observe that for any $R \leq k \leq 3R$, for which $k-R$ is even and any $y \in S_{k}(x)$ there is a path of length $2R$ between $y$ and a point in $S_R(x)$.
Again the parity assumption on $k-R$ can be dropped if we consider the paths to points in $S_R(x) \cup S_{R+1}(x)$ instead of $S_R(x)$.
Combining the above observations, we see that for every $R>0$ and every $y \in B_{3R+1}(x)$ there is a path of length $2R$ from $y$ to a point in $S_R(x) \cup S_{R+1}(x)$.
\end{remark}

For a fixed vertex $\widehat{x} = (x,n)$ of $\widehat{G}$ and positive integers $R$ and $L$, we consider several sets of vertices.
The following set of vertices is considered as the set of (lower) interior vertices:
\[ \mathcal{V}_{i}(\widehat{x};R,L) = \bigcup_{j=0}^{L}\mathring{B}_{R+j}(x)\times \{n+j\}.\]
The set of interior edges $\mathcal{E}_{i}(\widehat{x};R,L)$ is the set of all edges with at least one endpoint in $\mathcal{V}_{i}(\widehat{x};R,L)$ and at least one endpoint strictly above the level $n$.
We will also consider the lower and the upper boundary vertices:
\begin{itemize}
\item $\displaystyle \mathcal{V}_{lb}(\widehat{x};R,L) = \bigcup_{j=0}^{L}\bigl(S_{R+j}(x) \cup S_{R+j+1}(x)\bigr)\times \{n+j\}$,
\item $\displaystyle \mathcal{V}_{ub}(\widehat{x};R,L) = \bigcup_{j=1}^{2(R+L)}\bigl(B_{R+L+j+1}(x) \backslash \mathring{B}_{R+L-j}(x) \bigr)\times \{n+L+j\}$,
\end{itemize}
where $\mathring{B}_{r}(x) = \emptyset$, for $r \leq 0$.
The boundary vertex set is then $\mathcal{V}_{b}(\widehat{x};R,L) = \mathcal{V}_{lb}(\widehat{x};R,L) \cup \mathcal{V}_{ub}(\widehat{x};R,L)$, and the boundary edge set 
$\mathcal{E}_{b}(\widehat{x};R,L)$ is the set of edges whose both endpoints are in the set $\mathcal{V}_{b}(\widehat{x};R,L)$.

The following technical lemma provides the core of the blocking path argument.
We state and prove it separately, to make the proof of the main result more readable.
First observe that for a bipartite graph $G$ and two of it's vertices $x$ and $y$, either all the paths from $y$ to $S_R(x)$ have odd length, or all the paths from $y$ to $S_R(x)$ have even length (depending on the parity of $d(x,y)+R$).
In the latter case we will say that $y$ and $S_R(x)$ have the same parity.
Note that for any $R$ exactly one of $S_R(x)$ or $S_{R+1}(x)$ has the same parity as $y$.

\begin{lemma}\label{lemma:general_directed_killing_argument}
For a fixed vertex $\widehat{x} = (x,n) \in \widehat{G}$, and fixed positive integers $R$ and $L$, let $\alpha < \beta$ and $K$ be three positive real numbers satisfying
\begin{equation}\label{eq:general_directed_killing_numerical_condition}
L\beta  > (2R+3L)\alpha +K.
\end{equation}
Assume that $\omega(\widehat{e}) < \alpha$ for all $\widehat{e} \in \mathcal{E}_{b}(\widehat{x};R,L)$, and that $\omega(\widehat{e}) > \beta$ for all $\widehat{e} \in \mathcal{E}_{i}(\widehat{x};R,L)$.
Assume that for some $y \in \mathring{B}_R(x)$, the vertex $\widehat{y}= (y,n)$ satisfies $\widehat{d}_\omega(\widehat{y}',G) \leq \widehat{d}_\omega(\widehat{y},G) + K$, for all $\widehat{y}' = (y',n)$ of the form
\begin{itemize}
\item $y' \in S_R(x) \cup S_{R+1}(x)$, if $G$ is non-bipartite,
\item $y' \in S_R(x)$ if $G$ is bipartite and $y$ and $S_R(x)$ have the same parity,
\item $y' \in S_{R+1}(x)$ if $G$ is bipartite and $y$ and $S_{R+1}(x)$ have the same parity.
\end{itemize}
Then there is no vertex $\widehat{z}$ on the level $n+2R+3L$ whose geodesic $\widehat{\gamma}(\widehat{z})$ contains the point $\widehat{y}$.
\end{lemma}

\begin{proof}
If indeed there is a vertex $\widehat{z} = (z,n+2R + 3L)$ for which $\widehat{y} \in \widehat{\gamma}(\widehat{z})$ then we necessarily have $z \in \mathring{B}_{3R+3L}(x)$, so in particular $\widehat{z} \in \mathcal{V}_{ub}(\widehat{x};R,L)$.
If we denote the first $2R+3L$ vertices in the projection $\mathfs{P}(\widehat{\gamma}(\widehat{z}))$ by $z = z_0^1, z_1^1, \dots, z_{2R+3L}^1 = y$ then 
$z_{2R+3L-i}^1 \in B_{i}(y)$.
In particular, for $0 \leq i < L$ we have $z_{2R+3L-i}^1 \in \mathring{B}_{R+i}(x)$, so the last $L$ vertices (last $L$ edges) in the part of the geodesics $\widehat{\gamma}(\widehat{z})$ between $\widehat{z}$ and $\widehat{y}$ are in $\mathcal{V}_{i}(\widehat{x};R,L)$ ($\mathcal{E}_{i}(\widehat{x};R,L)$).
This yields
\begin{equation}\label{eq:general_directed_killing_numerical_condition_1}
\widehat{d}_\omega(\widehat{z},G) \geq \widehat{d}_\omega(\widehat{y},G) + L\beta.
\end{equation}
Since $z \in \mathring{B}_{3R+3L}(x)$, by Remark \ref{rem:dead_ends} one can find a path $\sigma$ in $G$ between $z$ and some $w \in S_{R+L}(x)\cup S_{R+L+1}(x)$ of length $2(R+L)$.
If we denote the vertices in the path $\sigma$ by $z = z_0^2, z_1^2, \dots, z_{2(R+L)}^2 = w$, it is clear that $z_{2(R+L)-i}^2 \in B_{i}(w)$.
In particular, this yields $z_{2(R+L)-i}^2 \in B_{R+L+i+1}(x) \backslash \mathring{B}_{R+L-i}(x)$.
Lift the path $\sigma$ to a path $\widehat{\sigma}$ from $\widehat{z}$ to $(w,n+L)$.
Now all the vertices in $\widehat{\sigma}$ are contained in the set $\mathcal{V}_{ub}(\widehat{x};R,L)$ so $\lambda(\widehat{\sigma}) \leq 2\alpha (R+L)$.
Consider a path of length $L$ which starts at $w$, always decreases the distance from $x$ and ends at $w' \in S_R(x) \cup S_{R+1}(x)$.
The $i$-th vertex of this path is in $S_{R+L-i}(x) \cup S_{R+L-i+1}(x)$.
Therefore, a lift of this path yields a path connecting $\widehat{w}$ to a point in $(S_R(x) \cup S_{R+1}(x)) \times \{n\}$ with all the vertices in $\mathcal{V}_{lb}(\widehat{x};R,L)$.
Using this path to extend $\widehat{\sigma}$ yields a path connecting $\widehat{z}$ and the point $\widehat{w}' = (w',n)$ of total $\lambda$ weight of at most $\alpha(2R+3L)$.
If $G$ is not bipartite then $\widehat{d}_\omega(\widehat{w}',G) \leq \widehat{d}_\omega(\widehat{y},G) + K$.
If $G$ is bipartite, then consider the union of the projections of the considered path between $\widehat{z}$ and $\widehat{y}$ and the constructed path between $\widehat{z}$ and $\widehat{w'}$.
This union defines a path between $y$ and $w'$ of even length $2(2R+3L)$, so $y$ and $w'$ are of the same parity.
Therefore, we again have $\widehat{d}_\omega(\widehat{w}',G) \leq \widehat{d}_\omega(\widehat{y},G) + K$.
This yields
\[
\widehat{d}_\omega(\widehat{z},G) \leq \widehat{d}_\omega(\widehat{y},G) + (2R+3L)\alpha - K.
\]
However, this combined with \eqref{eq:general_directed_killing_numerical_condition_1} yields the contradiction with \eqref{eq:general_directed_killing_numerical_condition}.

\end{proof}

\begin{proof}[Proof of Theorem \ref{thm:main} in the general directed case]
Without loss of generality we can assume that the bottom of the  support of $\omega(\widehat{e})$ is at 0.
It suffices to show that for each $x \in G$ we have $h(\widehat{T}(x))< \infty$ almost surely.
Assume the opposite $\widehat{\p}[h(\widehat{T}(x))=\infty] >0$ and fix $\epsilon < \widehat{\p}[h(\widehat{T}(x))=\infty]/4$.
Take $M>0$ such that $\widehat{\p}[|\widehat{T}^n(x)| > M] \leq \epsilon$.
For a fixed $R>0$ take $K_R>0$ so that 
\[\widehat{\p}[D_{\widehat{y},R;K_{R}}] \geq 1- \frac{\epsilon}{M+1}\] 
holds for every $\widehat{y}$. 
For a given $R$ fix a positive integer $L_R$, as well as positive real numbers $\alpha_R$ and $\beta_R$ so that the probabilities $\widehat{\p}[\omega(\widehat{e})>\beta_R]$ and $\widehat{\p}[\omega(\widehat{e})<\alpha_R]$ are both strictly positive and
\[
L_R\beta_R  > (2R+3L_R)\alpha_R +2K_R.
\]
This is possible since the lower edge of the support of $\omega(\widehat{e})$ is at 0.
Without loss of generality we can assume that $K_R$, $L_R$ and $\beta_R$  are all non-decreasing in $R$, while $\alpha_R$ is non-increasing in $R$.
Construct a sequence of even positive integers $(R_i)_{i \geq 1}$ as $R_1=2$ and $R_{i + 1} = 6(R_i + L_i+1)$, where $L_i = L_{R_i+1}$.
Also set $K_i = K_{R_i+1}$, as well as $\alpha = \alpha_{R_M+1}$ and $\beta = \beta_{R_M+1}$.
Now it is clear that for every $1 \leq i \leq M$,
\begin{equation}\label{eq:directed_general_right_cuttoffs}
L_i\beta  > (2R_i +2 +3L_i)\alpha +2K_i,
\end{equation}
Observe that by the union bound the event $\widetilde{D}_{\widehat{y}} = \cap_{i=1}^{M+1}D_{\widehat{y},R_i+1;K_i}$ has probability at least $1- \epsilon$ for any $\widehat{y}$.
Now apply the mass transport principle by sending a unit mass from $y$ to $x$ if $(y,n) \in \widehat{T}^n(x)$ and if the event $\widetilde{D}_{\widehat{y}}$ fails. 
By the above bound, the expected mass sent out of a vertex $y$ is at most $\epsilon$. 
The expected mass received by the vertex $x$ is an upper bound for the probability that there is a vertex $\widehat{y}=(y,n) \in \widehat{T}^n(x)$ for which $\widetilde{D}_{\widehat{y}}$ fails.
By the mass transport principle, we have
\[
\widehat{\p}[\widehat{T}^n(x) \neq \emptyset, \cup_{\widehat{y} \in \widehat{T}^n(x)}\widetilde{D}_{\widehat{y}}^c] \leq \epsilon.
\]
Consider the event
\[
A_n = \bigl(\cap_{\widehat{y} \in \widehat{T}^n(x)}\widetilde{D}_{\widehat{y}}\bigr) \bigcap \{1 \leq |\widehat{T}^n(x)| \leq M\}.
\]
Clearly $A_n \subset \{\widehat{T}^n(x) \neq \emptyset\}$, and the above bounds imply that $\widehat{\p}[\widehat{T}^n(x) \neq \emptyset] - \widehat{\p}[A_n] \leq 2\epsilon$.
Since $\epsilon < \widehat{\p}[\widehat{T}^n(x)\neq \emptyset]/4$, we obtain $\widehat{\p}[A_n] \geq \widehat{\p}[\widehat{T}^n(x) \neq \emptyset]/2$.
We will show that on the event $A_n$ we have 
\begin{equation}\label{eq:killing_estimate_general_directed}
\widehat{\p}[\widehat{T}^{n+2R_M + 3L_{M}}(x) = \emptyset|\mathfs{F}_n] \geq \delta > 0,
\end{equation}
where $\delta$ does not depend on the choice of $x$ or $n$.
Then by  Lemma \ref{lemma:killing_argument} the claim will follow.
Therefore, the rest of the proof is devoted to the proof of \eqref{eq:killing_estimate_general_directed}.

To end this we will construct a bounded number of $\mathcal{V}_i$, $\mathcal{V}_{lb}$ and $\mathcal{V}_{ub}$ type sets, of bounded radii, which are disjoint and enclose in their lower faces all the points from $\widehat{T}^n(x)$.
Then we will use the event from Lemma \ref{lemma:general_directed_killing_argument} to perform the path blocking.
To start, fix any ordering of the vertices of $G$.
On the event $A_n$ we necessarily have $1 \leq |\widehat{T}^n(x)| \leq M$, so by the pigeonhole principle for each $\widehat{y} = (y,n) \in \widehat{T}^n(x)$ we can find an index $1 \leq j_y \leq M$ so that for each $z \in \mathring{B}_{R_{j_y+1}}(y) \backslash \mathring{B}_{R_{j_y}}(y)$ we have $(z,n) \notin \widehat{T}^n(x)$ (take for example the smallest such $1 \leq j_y \leq M$ for each $y$).
Now order the vertices of $\mathfs{P}(\widehat{T}^n(x))$ in the decreasing order of the indices $j_y$, that is for $(y_1,n), (y_2,n) \in \widehat{T}^n(x)$ set $y_1 \preceq y_2$ if $j_{y_1} > j_{y_2}$. 
If $j_{y_1} = j_{y_2}$ then use the fixed ordering of the vertices in $G$, that is $y_1 \preceq y_2$ if $y_1$ comes before $y_2$ in the ordering of the vertices of $G$.
Let $y_1 \preceq y_2 \preceq \dots \preceq y_{|\widehat{T}^n(x)|}$ denote the above ordering of the points in $\widehat{T}^n(x)$.
Set $\mathfrak{S}_1 = \mathfs{P}(\widehat{T}^n(x)) = \{y_1 , y_2 , \dots, y_{|\widehat{T}^n(x)|}\}$.
Starting with $y_1'=y_1$, denote $R_1' = R_{j_{y_1'}}$ and 
remove from the set $\mathfrak{S}_1$ all the vertices which are contained in the ball $\mathring{B}_{R'_1}(y'_1)$ to construct the set $\mathfrak{S}_2$.
Now repeat this starting from $\mathfrak{S}_2$.
In general, given $\mathfrak{S}_i$ let $y_i'$ be the first element of $\mathfrak{S}_i$ with respect to $\preceq$ ordering, take $R_i' = R_{j_{y_i'}}$ 
and let  $\mathfrak{S}_{i+1} = \mathfrak{S}_i \backslash \mathring{B}_{R_i'}(y_i')$.
Stop the algorithm when $\mathfrak{S}_{k+1} = \emptyset$.
We will also use the notation $K_i'=K_{j_{y_i'}}$, $L_i'=L_{j_{y_i'}}$ and $\widehat{y}_i' = (y_i,n)$.

The algorithm returns $k$ points $y_1', \dots, y_k'$ in $\widehat{T}^n(x)$, with the corresponding radii $R_i'$.
Moreover, the union of open balls $\cup_{i=1}^k\mathring{B}_{R_i'}(y_i')$ cover the projection of $\widehat{T}^n(x)$, that is 
\[\widehat{T}^n(x) = \bigcup_{i=1}^k \bigl(\mathring{B}_{R_i'}(y_i') \times \{n\}\bigr).\]
Also observe that $R_1' \geq R_2' \geq \dots \geq R_k'$, and that, by the construction of the original sequence $(R_i)$ and the choice of the indices $j_{y}$ for $(y,n) \in \widehat{T}^n(x)$,  none of the annuli $\mathring{B}_{6(R_i'+L_i'+1)}(y_i') \backslash \mathring{B}_{R_i'}(y_i')$, $1 \leq i \leq k$ contain any point of $\mathfs{P}(\widehat{T}^n(x))$.

We claim that the balls $B_{3R_i'+3L_i'+1}(y_i')$, for $i=1,\dots, k$ are disjoint pairwise.
To show this, observe that for $i < j$ we necessarily have $y_j'\in \mathfrak{S}_i$, and in particular $y_j' \notin \mathring{B}_{R_i'}(y_i')$.
Since we also have $y_j' \notin \mathring{B}_{6(R_i'+L_i'+1)}(y_i') \backslash \mathring{B}_{R_i'}(y_i')$, it follows that $d(y_j',y_i') \geq 6(R_i'+L_i'+1)$.
Since $R_i' \geq R_j'$, and thus also $L_i' \geq L_j'$, it follows that the balls $B_{3R_i'+3L_i'+1}(y_i')$ and $B_{3R_j'+3L_j'+1}(y_j')$ are disjoint.

From this it is clear that the sets of vertices $\mathcal{V}_{i}(\widehat{y}_i';R_i'+1,L_i') \cup \mathcal{V}_{b}(\widehat{y}_i';R_i'+1,L_i')$, $1 \leq i \leq k$ are pairwise disjoint. 
The same is then also true for the edges sets $\mathcal{E}_{i}(\widehat{y}_i';R_i'+1,L_i') \cup \mathcal{E}_{b}(\widehat{y}_i';R_i'+1,L_i')$, $1 \leq i \leq k$.
Therefore, given $A_n$, one can define the event $E_n$ for which 
\begin{itemize}
\item $ \displaystyle  \omega(\widehat{e}) < \alpha$ for all $\displaystyle \widehat{e} \in \bigcup_{i=1}^k \mathcal{E}_{b}(\widehat{y}_i';R_i'+1,L_i')$, and
\item $ \displaystyle  \omega(\widehat{e}) > \beta$ for all $ \displaystyle  \widehat{e} \in \bigcup_{i=1}^k \mathcal{E}_{i}(\widehat{y}_i';R_i'+1,L_i')$.
\end{itemize}
Since the event $E_n$ puts constraints on at most $M(2R_M+3L_M)\overline{N}$ vertices, where $\overline{N}$ is the number of vertices in a ball of radius $3R_M + 3L_M +1$ in $G$, it is clear that on the event $A_n$ we have that $\widehat{\p}[E_n|\mathfs{F}_n]$ is bounded away from 0, as $n \to \infty$.
To show \eqref{eq:killing_estimate_general_directed} and finish the proof it suffices to show that on the event $A_n \cap E_n$ we necessarily have  $\widehat{T}^{n+2R_M + 3L_{M}}(x) = \emptyset$.
Assume the opposite, that $A_n \cap E_n$ happens and  $\widehat{T}^{n+2R_M + 3L_{M}}(x) = \emptyset$ holds. 
Then there would have to be a geodesic from $G_{n+2R_M + 3L_{M}}$ to $G$ passing through some point $(z,n)$ of $\widehat{T}^n(x)$, and let $1 \leq i \leq k$ be the index such that $(z,n) \in \mathring{B}_{R_i'}(y_i')$.
Note that if $G$ is bipartite we necessarily have that $d(y_i',z)$ is even, since both $(y_i',n)$ and $(z,n)$ are in $\widehat{T}^n(x)$. 
Since $R_i'$ is also even, if $G$ is bipartite, we have that $z$ and $S_{R_i'}(y_i')$ are of the same parity.
We will reach the contradiction, by applying Lemma \ref{lemma:general_directed_killing_argument}, for $\widehat{x} := \widehat{y}_i'$, $R:= R_i'$, $L:=L_i'$, $K:=2K_i'$ and $\alpha$ and $\beta$ as in the current proof.
To justify the application observe that the condition \eqref{eq:directed_general_right_cuttoffs} immediately implies \eqref{eq:general_directed_killing_numerical_condition}. 
On the other hand, the event $A_n$ implies $\widetilde{D}_{\widehat{y_i'}}$, and so in particular $D_{\widehat{y_i'},R_i'+1;K_i'}$ holds. 
Therefore, if $G$ is non-bipartite, we have $\widehat{d}_\omega((z',n),G) \leq \widehat{d}_\omega((z,n),G) + 2K_i'$ for every $z' \in S_{R_i'}(y_i') \cup S_{R_i'+1}(y_i')$. If $G$ is bipartite, then we can conclude the same for every $z' \in S_{R_i'}(y_i')$ since $R_i'$ is even (as we argued above $z$ and $S_{R_i'}(y_i')$ are of the same parity).

\end{proof}

\subsection{Undirected case for general graphs}

Let $S$ be a set of vertices at the level $n,$ that is $S \subset G_n$.
We define $Cl^v_n(S)$, the level $n$ vertex closure of $S$, as the set of vertices of the form $(y,n)$ which are either elements of $S$, or have a neighbor in $S$.
The level $n$ edge closure $Cl^e_n(S)$ of $S$ is defined as the set of edges connecting two vertices in $Cl^v_n(S)$, and the level $n$ vertex boundary $\partial^v_nS$ of $S$ is defined as $Cl^v_n(S)  \backslash S$.
For $L>0$ define the cylinder $\mathcal{C}_L(S)$ with the base $S$ and height $L$ as the subgraph of $\overline{G}$ whose vertices are $(y,k)$ for $y$ such that $(y,n) \in Cl^v_n(S)$ and $k$ such that $n \leq k \leq n+L$.
The edge set of $\mathcal{C}_L(S)$ consists of all edges of $\overline{G}$ which connect any two vertices of $\mathcal{C}_L(S)$.
Furthermore, we define the side vertex and edge boundary of $\mathcal{C}_L(S)$ as 
\[
\partial_s^v \mathcal{C}_L(S) = \{(y,k) : y\in \mathfs{P}(\partial^v_n S), n < k \leq n+L \} 
\]
and
\[
\partial_s^e \mathcal{C}_L(S) = \{(y^\downarrow,k) : y\in \mathfs{P}(\partial^v_n S), n < k \leq n+L \},
\]
respectively. {Here $(y^\downarrow,k)$ denotes the edge $((y,k),(y,k-1))$.}
The top vertex and edge boundary of $C_L(S)$ as 
\[
\partial_t^v \mathcal{C}_L(S) = \{(y,n+L) : y\in  \mathfs{P}(Cl_n^v(S)) \} 
\]
and
\[
\partial_t^e \mathcal{C}_L(S) = \{(e,n+L) : e \in \mathfs{P}(Cl_n^e(S)\},
\]
respectively. Here $(e,n+L)$ denotes the edge $e$ in $G$ lifted to level $n+L$.
Finally the interior vertex set $In^v(\mathcal{C}_L(S))$ and the interior edge set $In^e(\mathcal{C}_L(S))$ are defined as the set of vertices of the form $\{(y,k) : y \in S, n < k < n+L\}$, and the set of edges of $\mathcal{C}_L(S)$ which are not in $Cl_n^e(S) \cup \partial_s^e \mathcal{C}_L(S) \cup \partial_t^e \mathcal{C}_L(S)$, respectively.

\begin{proof}[Proof of Theorem \ref{thm:main} in the general undirected case]
Assume  $\overline{\p}[|\overline{T}(x)| = \infty] > 0$, so that 
\[
\overline{\p}\Big[\bigcap_n \{\overline{T}^n(x) \neq \emptyset\}\Big] = \overline{\p}[|\overline{T}(x)| = \infty] > 0.
\]
By  Lemma \ref{lemma:ergodic} we have for $M$ large enough 
\begin{align*}
\overline{\p}\big[1 \leq |\overline{T}^n(x)| \leq M\big] & = \overline{\p}\big[\overline{T}^n(x) \neq \emptyset\big]- \overline{\p}\big[|\overline{T}^n(x)| \geq M\big]  \\
& \geq \overline{\p}\big[|\overline{T}(x)|  = \infty\big] - \frac{1}{M} \geq \frac{1}{2}\overline{\p}\big[|\overline{T}(x)|  = \infty\big].
\end{align*}
Given such a value of $M$.
For $\kappa > 0$ and a fixed vertex $x$ consider the event 
\[
D_n = \{1 \leq |\overline{T}^n(x)| \leq M , \omega(\overline{e}) < \kappa, \text{ for all } \overline{e} \in Cl_n^e(\overline{T}^n(x))\}.
\]
We apply \eqref{eq:ergodic_with_events_directed} in Lemma \ref{lemma:ergodic_with_conditioning_general} in the case when $A_S$ is the event defined as $\omega(\overline{e}) < \kappa$, for all 
$\overline{e} \in Cl_n^e(S)$.
We obtain
\[
\overline{\p}[D_n] \geq \overline{\p}[1 \leq |\overline{T}^n(x)| \leq M] - \overline{\p}[A_{B_M(x) \times n}^c].
\]
Since $M$ is fixed, by taking $\kappa$ large enough, but still smaller than the maximum of the support of $\omega(\overline{e})$, we can make the term $\overline{\p}[A_{B_M(x) \times n}^c]$ arbitrarily small, so that
\[
\overline{\p}[D_n] \geq \frac{\overline{\p}[|\overline{T}(x)|  = \infty]}{3},
\]
for all $n$.
Fix such  a value of $\kappa$ and  take two values $\kappa_1$ and $\kappa_2$ such that $\kappa_1 < \kappa < \kappa_2$ and such that $\overline{\p}[\omega(\overline{e}) < \kappa_1] >0$ and $\overline{\p}[\omega(\overline{e}) > \kappa_2] >0$, and take an integer $L$ such that 
\[
L \geq \frac{3M\kappa}{\kappa_2-\kappa_1}.
\]
Consider the event $S_{n,L}$ for which $\omega(\overline{e}) < \kappa_1$ for all $\overline{e} \in \partial_s^e \mathcal{C}_L(\overline{T}^n(x)) \cup \partial_t^e \mathcal{C}_L(\overline{T}^n(x))$ and $\omega(\overline{e}) > \kappa_2$ for all $\overline{e} \in  In^e(\mathcal{C}_L(\overline{T}^n(x)))$.
If $1 \leq |\overline{T}^n(x)| \leq M$ the event $S_{n,L}$ puts constraints on at most  $(3d+2)ML$ edges, where $d$ is the degree in graph $G$. 
Moreover, these constraints are on the weights of edges outside of $\Lambda_n$, so on the event $D_n$ we have $\overline{\p}[S_{n,L}|\mathfs{F}_n] \geq \delta$, for some $\delta > 0$ and all $n$.
By Lemma \ref{lemma:killing_argument}, it is  sufficient to show that on the event $D_n \cap S_{n,L}$ we have that $\overline{T}^{n+L}(x) = \emptyset$, and the rest of the proof is dedicated to this.

Assume that each of the events $D_n$, $S_{n,L}$ and $\overline{T}^{n+L}(x)  \neq \emptyset$ happen.
Choose a vertex $\overline{y}_0 = (y,n+L) \in \overline{T}^{n+L}(x)$ such that the geodesic $\overline{\gamma}_{\Lambda_{n+L}}(\overline{y}_0)$ has the smallest number of edges among all $\overline{\gamma}_{\Lambda_{n+L}}$ geodesics  from a vertex in $\overline{T}^{n+L}(x)$ to $G$.
Label the vertices in $\overline{\gamma}_{\Lambda_{n+L}}(\overline{y}_0)$ starting from $\overline{y}_0$ by $\overline{y}_0,\overline{y}_1, \dots, \overline{y}_m = x$.
Let $k \geq 1$ be the smallest index such that $\overline{y}_i \in \Lambda_n$, for all $i \geq k$, and denote $\overline{y}_k = (y_k,n)$.
It is easily observed that $\overline{y}_k \in \overline{T}^n(x)$ and $\overline{d}_{\omega,\Lambda_{n+L}}(\overline{y}_k) = \overline{d}_{\omega,\Lambda_n}(\overline{y}_k)$. Furthermore, the vertex $\overline{y}_{k-1}$ is the neighbor above the vertex $\overline{y}_k$,  so in particular $\overline{y}_{k-1} \in In^v(\mathcal{C}_L(\overline{T}^n(x)))$. 
We claim that we can find an index $l$, $ 0\leq l < k-1$ such that $\overline{y}_l \in \partial_t^v \mathcal{C}_L(\overline{T}^n(x)) \cup \partial_s^v   \mathcal{C}_L(\overline{T}^n(x)) \cup Cl_n^v(\overline{T}^n(x))$ and that $\overline{y}_{i} \in In^v(\mathcal{C}_L(\overline{T}^n(x)))$ for all $l < i <k$.
Observe that it suffices to show that there is some index $l$ such that $ 0\leq l < k-1$ and 
\begin{equation}\label{eq:last_time_hitting_cyllinder_general}
\overline{y}_l \in \partial_t^v \mathcal{C}_L(\overline{T}^n(x)) \cup \partial_s^v   \mathcal{C}_L(\overline{T}^n(x)) \cup Cl_n^v(\overline{T}^n(x)),
\end{equation} 
then we simply take $l$ to be the largest such index.
If for all $0 \leq i \leq k$ we have $\overline{y}_i \in \mathcal{C}_L(\overline{T}^n(x))$, then necessarily  $\overline{y}_0 \in \mathcal{C}_L(\overline{T}^n(x)) \cap G_{n+L} = \partial_t^v \mathcal{C}_L(\overline{T}^n(x))$, so \eqref{eq:last_time_hitting_cyllinder_general} is satisfied for $l=0$. 
Otherwise, take $j$ to be the largest index smaller than $k$ such that $\overline{y}_j \notin \mathcal{C}_L(\overline{T}^n(x))$, and observe that then \eqref{eq:last_time_hitting_cyllinder_general} is satisfied for $l=j+1$. 

In either case we proved the existence of the index $l$ satisfying \eqref{eq:last_time_hitting_cyllinder_general} and $\overline{y}_{i} \in In^v(\mathcal{C}_L(\overline{T}^n(x)))$ for all $l < i <k$.
Now we argue that the event $D_n \cap S_{n,L}$ implies that  $\overline{y}_l \notin Cl_n^v(\overline{T}^n(x))$, so that $\overline{y}_l \in \partial_t^v \mathcal{C}_L(\overline{T}^n(x)) \cup \partial_s^v \mathcal{C}_L(\overline{T}^n(x))$.
Assume that $\overline{y}_l  \in Cl_n^v(\overline{T}^n(x))$ and project the path $\overline{y}_l,\overline{y}_{l+1}, \dots, \overline{y}_k$ onto $G$ to obtain a path in $G$ whose vertices we denote by $z_0, z_1, \dots , z_r$ with $r \leq k-l$.
The path connecting vertices $\overline{y}_l = (z_0,n), (z_1,n), \dots , (z_r,n) = \overline{y}_k$ has weight at most $r\kappa$.
The sequence $\overline{y}_l, \overline{y}_{l+1}, \dots, \overline{y}_k$ defines a path consisting of exactly $r$ ``horizontal'' edges in $In^e(\mathcal{C}_L(\overline{T}^n(x)))$.
The total sum of weights of these edges is at least $r\kappa_2 > r\kappa$.
In particular the path  $(z_0,n), (z_1,n), \dots , (z_r,n)$ connecting $\overline{y}_l$ and $\overline{y}_k$ is lighter than the part of the geodesic connecting these two vertices, which yields a contradiction.
Therefore we showed that $\overline{y}_l \in \partial_t^v \mathcal{C}_L(\overline{T}^n(x)) \cup \partial_s^v \mathcal{C}_L(\overline{T}^n(x))$.

Let $j$ be the height of $\overline{y}_l$, that is $\overline{y}_l = (z_0,j)$.
The path $\overline{y}_l, \overline{y}_{l+1}, \dots, \overline{y}_k$ consists of exactly $r$ horizontal edges and at least $j-n$ vertical edges, all of which are in $In^e(\mathcal{C}_L(\overline{T}^n(x)))$.
On the event $D_n \cap S_{n,L}$ the total weight of this path is at least $(r+j-n)\kappa_2$, and since this path 
is a part of the geodesic through $\overline{y}_l$ we obtain
\begin{equation}\label{eq:connecting_by_the_geodesic_undirected_general}
\overline{d}_{\omega,\Lambda_{n+L}}(\overline{y}_l,G) \geq \overline{d}_{\omega,\Lambda_{n+L}}(\overline{y}_k,G) + (r+j-n)\kappa_2.
\end{equation}

Consider the first case when $\overline{y}_l = (z_0,j) \in  \partial_s^v \mathcal{C}_L(\overline{T}^n(x))$, so that in particular $(z_0,n) \in \partial_n^v\overline{T}^n(x)$.
Comparing the geodesic $\overline{\gamma}_{\Lambda_{n+L}}(z_0,n)$ from $(z_0,n)$ to $G$ with the concatenation of the path  $(z_0,n), (z_1,n), \dots , (z_r,n)$ and the geodesic $\overline{\gamma}_{\Lambda_{n+L}}(\overline{y}_k)$ from $\overline{y}_k$ to $G$, we have
\[
\overline{d}_{\omega,\Lambda_{n+L}}((z_0,n),G) \leq  \overline{d}_{\omega,\Lambda_{n+L}}(\overline{y}_k,G)+r\kappa.
\] 
Furthermore, the vertical path connecting edges $\overline{y}_l$ and $(z_0,n)$ has weight at most $(j-n)\kappa_1$, which gives 
\[
\overline{d}_{\omega,\Lambda_{n+L}}(\overline{y}_l,G) \leq \overline{d}_{\omega,\Lambda_{n+L}}((z_0,n),G) + (j-n)\kappa_1 \leq \overline{d}_{\omega,\Lambda_{n+L}}(\overline{y}_k,G)+r\kappa + (j-n)\kappa_1.
\]
This is a contradiction with \eqref{eq:connecting_by_the_geodesic_undirected_general}.

Now assume that $\overline{y}_l \notin  \partial_s^v \mathcal{C}_L(\overline{T}^n(x))$, so that $\overline{y}_l \in  \partial_t^v \mathcal{C}_L(\overline{T}^n(x))$ and $(z_0,n) \in \overline{T}^n(x)$. 
In this case \eqref{eq:connecting_by_the_geodesic_undirected_general} is satisfied for $j=n+L$ and $r=0$, that is
\begin{equation}\label{eq:connecting_by_the_geodesic_undirected_general-top}
\overline{d}_{\omega,\Lambda_{n+L}}(\overline{y}_l,G) \geq \overline{d}_{\omega,\Lambda_{n+L}}(\overline{y}_k,G) + L\kappa_2.
\end{equation}
Take a vertex $z'$ in $\mathfs{P}(\partial_n^v\overline{T}^n(x))$ such that $(z',n)$ can be connected to both $(z_0,n)$ and $(z_r,n) = \overline{y}_k$ by paths with edges in $Cl_n^e(\overline{T}^n(x))$ of lengths at most $M$ and  $2M$, respectively.
To construct this vertex observe that $(z_0,n) \in \overline{T}^n(x)$, $|\overline{T}^n(x)| \leq M$ and the ball $B_M(z_0)$ has more than $M$ elements. 
Therefore  we can actually find a vertex $(z',n)\in \partial_n^v \overline{T}^n(x)$ which can be connected to $(z_0,n)$ by a path in $Cl_n^e(\overline{T}^n(x))$ of length at most $M$.
Extending this path through $(z_0,n), \dots (z_r,n)$ creates a path in $Cl_n^e(\overline{T}^n(x))$ between $(z',n)$ and $(z_r,n) = \overline{y}_k$ of length at most $2M$.
Again, comparing the geodesic $\overline{\gamma}_{\Lambda_{n+L}}(z',n)$  with the concatenation of the above path between $(z',n)$ and $\overline{y}_k$, and the geodesic $\overline{\gamma}_{\Lambda_{n+L}}(\overline{y}_k)$ now gives
\[
\overline{d}_{\omega,\Lambda_{n+L}}((z',n),G) \leq \overline{d}_{\omega,\Lambda_{n+L}}(\overline{y}_k,G) + 2M\kappa.
\]
Combined with \eqref{eq:connecting_by_the_geodesic_undirected_general-top} this yields
\[
\overline{d}_{\omega,\Lambda_{n+L}}(\overline{y}_l,G) \geq \overline{d}_{\omega,\Lambda_{n+L}}((z',n),G)+ L\kappa_2 -2M\kappa.
\]
However, we will obtain a contradiction by showing that  $\overline{y}_l$ and $(z',n)$ can be connected by a path whose total weight is strictly less that $L\kappa_2 -2M\kappa$.
To construct this path lift the path in $Cl_n^e(\overline{T}^n(x))$ of length at most $M$ connecting $(z_0,n)$ and $(z',n)$ by height $L$, to obtain a path in $\partial_t^e\mathcal{C}_L(\overline{T}^n(x))$ connecting $\overline{y}_l$ and $(z',n+L)$.
On the event $D_n \cap S_{n,L}$ this path has total weight at most $M\kappa_1$.
Then connect the points $(z',n+L)$ and $(z',n)$ by a vertical path which contains $L$ edges in $\partial_s^e\mathcal{C}_L(\overline{T}^n(x))$, and whose total weight is thus at most $L\kappa_1$.
Concatenating these two paths gives a path between $\overline{y}_l$ and $(z',n)$ of total weight at most $(M+L)\kappa_1$.
By the choice of $L$ we have $(M+L)\kappa_1 < L\kappa_2 -2M\kappa$, which finishes the proof.

\end{proof}
\section{Tree height and shape}\label{sec:Tree height and shape}
In this section, we first prove Theorems \ref{thm:finite_widhts} and \ref{thm:finite_heights}. 
Given Theorem \ref{thm:main}, the first proof is rather short.
\begin{proof}[Proof of Theorem \ref{thm:finite_widhts}]
By Theorem \ref{thm:main} we have that $\lim_{n\rightarrow\infty}w_n(T(x))=0$ a.s. If we assume that $\e[w(T(x))]<\infty$, since $w_n(T(x))\le w(T(x))$, we obtain that $\e[w_n(T(x))]<\infty$. By the dominated convergence theorem, $\e[w_n(T(x))]\rightarrow 0$. This is a contradiction to Lemma \ref{lem:expwidthn}.
\end{proof}


\begin{proof}[Proof of Theorem \ref{thm:finite_heights} i)]
The proof follows easily from the observation that $w(\widehat{T}(x))\le \phi_G(h(\widehat{T}(x)))$, and Theorem \ref{thm:finite_widhts}.
\end{proof}

For the proofs of Theorem \ref{thm:finite_heights} in the undirected case and Theorem \ref{thm:shape}, we are using results about full lattice first passage percolation by Kesten \cite{kesten1993speed}. For the sake of readability, we will use a weak interpretation of Kesten's result. 
Let $\prob^{\BZ^2}$ be the product measure of $\mu$ over the edges of the full planar lattice $\BZ^2$. Let $B^{\BZ^2}(x,t)$ be the ball of radius $t$ around $x$ in the first passage percolation metric induced by $\prob^{\BZ^2}$ i.e. $B^{\BZ^2}(x,t)=\{y\in\BZ^2:d^{\BZ^2}_\omega(y,x)<t\}$, where $d^{\BZ^2}_\omega(y,x)$ is the minimal weight of any nearest neighbor path in $\BZ^2$ connecting $x$ and $y$. Abbreviate $\overline{B}(x,t)=\{y\in\overline{\BH}:\overline{d}_\omega(y,x)<t\}.$ Since the measure $\overline{\prob}$ is a restriction of $\prob^{\BZ^2}$, we abuse notation and consider $B(x,t)$ measurable events under the measure $\prob^{\BZ^2}$.

\begin{proof}[Proof of Theorem \ref{thm:finite_heights} ii)] For simplicity we write the proof for the case $G=\BZ$, the case $G=\BZ^d$ follows the same proof idea.
Assume for the sake of contradiction that $\overline{\ev}[h(\overline{T}(0))]<\infty$. Then
\bae\label{eq:height_moment_contradiction}
\infty>\overline{\ev}[2h(\overline{T}(0))]=\sum_{x=0}^\infty\overline{\prob}\left[h(\overline{T}(0))>\frac{x}{2}\right]>\frac{1}{2}\sum_{x=-\infty}^{\infty}\overline{\p}\left[h(\overline{T}(x))>\frac{|x|}{2}\right]
.\eae
By Borel - Cantelli we obtain that $\overline{\prob}$-a.s.~there exists some $R>0$ such that $h(\overline{T}(x))<\frac{|x|}{2}$ for all $|x|>R$.
The contradiction is reached by showing that in fact for all but finitely many $x \in \mathbb{Z}$ the tree $\overline{T}(x)$ is contained in  the Euclidean ball of radius $\frac{2|x|}{3}$ around $x$, and thus $\overline{\BH}$ is not covered.
The rest of the proof is devoted to this.

The key ingredient in the proof is Kesten's result \cite[Theorem 2]{kesten1993speed}.
For any vector of Euclidean length 1, $\|\xi\|_2=1$, let $m(\xi)=\lim_{n\rightarrow\infty}\frac{d^{\mathbb{Z}^2}_\omega(\lfloor n\xi\rfloor,0)}{n}$, where for $(u,v)\in\BR^2$, $\lfloor(u,v)\rfloor=(\lfloor u\rfloor,\lfloor v\rfloor)$. It is known (by convexity) that for every unit vector $\rho\in\BR^2$, $m(\rho)\ge m(0,1)$.

Let $x\in\BZ$ s.t, $\htr{x} \leq \frac{|x|}{2}$. Consider a vertex $y=(u,v)\in\overline{\BH}$ such that $v \leq \frac{|x|}{2}$ and $\|y-x\|_2>\frac{2|x|}{3}$.
By \cite[Theorem 2]{kesten1993speed}, for $x$ large enough we have \[\overline{d}_\omega (y,x) \geq d^{\mathbb{Z}^2}_\omega(y,x)>\frac{7|x|}{12}m\left(\frac{y-x}{\|y-x\|_2}\right) \geq \frac{7|x|}{12}m\left(0,1\right)\] with probability higher than $1-e^{-c|x|^{1/4}}$. On the other hand with the same probability $d^{\mathbb{Z}^2}_\omega(y,(u,0))<\frac{13|x|}{24}m(0,1)$. 
The geodesic between $y$ and $(u,0)$ in $\mathbb{Z}^2$ is either contained in $\overline{\BH}$ and then $\overline{d}_\omega(y,(u,0))<\frac{13|x|}{24}m(0,1)$, or it makes the first intersection with $\partial \overline{\BH}$ at some $(u_1,0)$, and then $\overline{d}_\omega(y,(u_1,0))<\frac{13|x|}{24}m(0,1)$.
Thus with probability higher than $1-2e^{-c|x|^{1/4}}$, it holds that $y\notin \overline{T}(x)$. By Borel-Cantelli all but finite number of trees $\overline{T}(x)$ are contained in the Euclidean ball of radius $\frac{2|x|}{3}$ around $x$.
\end{proof}

\begin{remark}\label{rem:question_about_moments}
Note that the above result leaves the possibility that for $G = \mathbb{Z}^d$ the $k$-th moments of  $h(\overline{T}(0))$  are finite for $1 \leq k \leq d-1$. 
At this points it is not clear what is the smallest value of $k$ for which $\overline{\mathbf{E}}[h(\overline{T}(0))^k] = \infty$, or how tree heights behave as the dimension $d$ increases.
\end{remark}

The next lemma is a weaker version of results found in \cite{ahlberg2013convergence}. We present a proof for the sake of completeness.
\begin{lemma}\label{lem:bijec} For every $t>0$, $\prob^{\BZ^2}$-a.s.
$$\bigcup_{x\in\BZ}\overline{T}(x,t)=\bigcup_{x\in\BZ}\overline{B}(x,t)=\bigcup_{x\in\BZ}B^{\BZ^2}(x,t).$$
\end{lemma}
\begin{proof}
First note that for every $x\in\BZ$ and $t>0$,
$$\overline{T}(x,t)\subset \overline{B}(x,t)\subset B^{\BZ^2}(x,t).$$
For the other direction, let $y\in B^{\BZ^2}(x,t)$. 
It suffices to show that $y \in \overline{B}(x',t)$ for some $x' \in \BZ$. 
Then, if $y \in \overline{T}(x',t)$ we are done, and otherwise $y \in \overline{T}(x'',t)$ for some $x'' \in \mathbb{Z}$ with $\overline{d}_\omega (x'',y)  \leq \overline{d}_\omega (x',y)   < t$, so in particular $y \in \overline{B}(x'',t)$.
To show the existence of such $x'$ observe that there exists a path from $x$ to $y$, with edges $e_1,\ldots,e_n$ in $\BZ^2$ connecting $x$ and $y$, satisfying $\sum_{i=1}^n\omega(e_i)<t$. 
If this path is contained in $\overline{\BH}$, then it is clear that $y \in \overline{B}(x,t)$ and set $x' = x$. 
Assuming that this path is not contained in $\overline{\BH}$, consider the last edge in $\overline{\BH}^c$ and abbreviate it $e_l$. 
If $x'$ is the vertex in $\mathbb{Z}$ incident to both $e_l$ and $e_{l+1}$, then the path $e_{l+1},\ldots,e_n$ is contained in $\overline{B}(x',t)$, that is $y\in \overline{B}(x',t)$ and we are done. 
\end{proof}
We can now prove Theorem \ref{thm:shape}.
\begin{proof}[Proof of Theorem \ref{thm:shape}]
By Lemma \ref{lem:bijec}
\bae
\left(\partial^{in}\bigcup_{x\in\BZ}\overline{T}(x,t)\right)\cap\left([-t,t]\times\BZ_+\right)=
\left(\partial^{in}\bigcup_{x\in\BZ}B^{\BZ^2}(x,t)\right)\cap\left([-t,t]\times\BZ_{+}\right)
.\eae
Therefore, we can replace $\left(\partial^{in}\bigcup_{x\in\BZ}\overline{T}(x,t)\right)\cap\left([-t,t]\times\BZ_+\right)$ in the definition of the event $C_{\mathbf{d},t}$ by $\left(\partial^{in}\bigcup_{x\in\BZ}B^{\BZ^2}(x,t)\right)\cap\left([-t,t]\times\BZ_{+}\right)$.
By \cite[Theorem 2]{kesten1993speed}, there exists some compact convex set $\CD'$ and $c>0$ such that for $\CD_t^1 = ((1-t^{-0.1})t\CD')\cap \mathbb{Z}^2$ and $\CD_t^2 = ((1+t^{-0.1})t\CD')\cap \mathbb{Z}^2$
\bae
\prob^{\BZ^2}\left[\CD_t^1 \subset B^{\BZ^2}(0,t)\subset \CD_t^2\right]\ge 1-e^{-ct^{1/4}}.
\eae
Now let $\mathbf{d}$ be the highest point of $\CD'$ i.e. $\mathbf{d}=\max\{y\in\BR:\exists x\in \BR, (x,y)\in\CD'\}$. 
Consider the box $\CB_t=([-3\mathbf{d}t,3\mathbf{d}t] \times [0,3\mathbf{d}t])\cap\BZ^2$, thus by union bound,
\bae\label{eq:goodballs}
\prob^{\BZ^2}\left[\forall x\in\CB_t,~x+\CD_t^1\subset B^{\BZ^2}(x,t) \subset x+\CD_t^2\right]\ge 1-\frac{1}{2}e^{-ct^{1/5}}
.\eae
Furthermore, consider the event that for every $x \in \CB_t' = [-t,t] \times [2\mathbf{d}t,\infty)$ the ball $B^{\mathbb{Z}^2}(x,t)$ does not intersect $\partial \overline{\BH}$.
Then the probability of this event is bounded by
\begin{equation}\label{eq:goodballs_2}
\mathbf{P}^{\BZ^2}\Big[
\forall x \in \CB_t',~B^{\mathbb{Z}^2}(x,t) \cap \partial \overline{\BH} = \emptyset 
\Big]
\geq 
1- 2t\sum_{s > 2t\mathbf{d}}e^{-cs^{1/4}} \geq 1-\frac{1}{2}e^{-ct^{1/5}}.
\end{equation}

Assume that both the events in \eqref{eq:goodballs} and \eqref{eq:goodballs_2} hold. 
Let \[y=(v,w)\in\left(\partial^{in}\bigcup_{x\in\BZ}B^{\BZ^2}(x,t)\right)\cap\left([-t,t]\times\BZ_{+}\right).\] 
In particular $B^{\mathbb{Z}^2}(y,t)$ intersects  $\partial \overline{\BH}$.
By the event in \eqref{eq:goodballs} it is obvious that $w > \mathbf{d}t(1-t^{-0.1})$, because otherwise the point $y$ would be in the interior of the ball $B^{\BZ^2}((v,0),t)$.
If $w > 2\mathbf{d}t$, then by the event in \eqref{eq:goodballs_2} the set $B^{\mathbb{Z}^2}(y,t)$ does not intersect $\partial \overline{\BH}$, which gives the contradiction.
Therefore, assume that  $w \leq 2\mathbf{d}t$.
If $w \geq \mathbf{d}t(1+t^{-0.1})$, then by the event in \eqref{eq:goodballs} again the set $B^{\mathbb{Z}^2}(y,t)$ does not intersect $\partial \overline{\BH}$, which again yields the contradiction. 
This finished the proof.


\end{proof}

\begin{remark}\label{rem:Ahlberg}
	Note that by the recent results of Ahlberg \cite{ahlberg2014temporal}, one can relax the condition \eqref{eq:kestenreq} in the undirected case of Theorem \ref{thm:finite_heights} and in Theorem \ref{thm:shape}.
	Namely it is sufficient to assume that 
	$\e[\omega(e)^{d+1+\epsilon}]< \infty$ for any $\epsilon > 0$.
\end{remark}

\begin{remark}\label{rem:shape}
As for the shape of $\bigcup_x T(x,t)$, the proof of  Theorem \ref{thm:shape} trivially generalizes to  higher dimensions, that is $G= \BZ^d$.
More precisely, there is some $\mathbf{d} > 0$, such that the set 
\[
\left(\partial^{in}\bigcup_{x\in\BZ^d}\overline{T}(x,t)\right)\cap\left([-t,t]^d\times\BZ_+\right) 
\]
is contained in 
\[
[-t,t]^d\times t[\mathbf{d}-2t^{-1/(2d+5)},\mathbf{d}+2t^{-1/(2d+5)}],
\]
with probability of at least $1 - e^{-ct^{1/5}}$, for some $c>0$.
\end{remark}

\section*{Acknowledgments}
We would like to thank Gideon Amir for pointing out the work of Deijfen and H{\"a}ggstr{\"o}m \cite{deijfen2007two}, and Yuval Peres for suggesting to use the mass transport principle which strengthened our results.
We would also like to thank Kenneth Alexander and Christopher Hoffman for useful discussions.
\bibliography{IDLAT}
\bibliographystyle{plain}

\end{document}